\documentclass[11pt, a4paper]{amsart}
\usepackage{mathbbol}
\usepackage{amsmath, amsthm, amssymb}
\usepackage{txfonts, mathrsfs}
\usepackage[utf8]{inputenc}
\usepackage[usenames]{color}
\usepackage{graphicx}

\newtheorem*{thm*}{Theorem}

\newtheorem{thm}{Theorem}[section]
\newcommand{\bt}{\begin{thm}}
\newcommand{\et}{\end{thm}}

\newtheorem{cor}[thm]{Corollary}
\newcommand{\bc}{\begin{cor}}
\newcommand{\ec}{\end{cor}}

\newtheorem{lem}[thm]{Lemma}
\newcommand{\bl}{\begin{lem}}
\newcommand{\el}{\end{lem}}

\newtheorem{prop}[thm]{Proposition}
\newcommand{\bp}{\begin{prop}}
\newcommand{\ep}{\end{prop}}

\newtheorem{defn}[thm]{Definition}
\newcommand{\bd}{\begin{defn}}      
\newcommand{\ed}{\end{defn}}

\newtheorem{rmrk}[thm]{Remark}
\newcommand{\br}{\begin{rmrk}}
\newcommand{\er}{\end{rmrk}}

\newtheorem{quest}[thm]{Question}
\newcommand{\bq}{\begin{quest}}
\newcommand{\eq}{\end{quest}}

\newcommand{\N}{\mathbb{N}}

\newcommand{\R}{\mathbb{R}}
\newcommand{\Z}{\mathbb{Z}}

\newcommand{\relmiddle}[1]{\mathrel{}\middle#1\mathrel{}}

\newdimen\vintkern\vintkern12pt
\def\vint{-\kern-\vintkern\int}

\newcommand{\hm}{{\mathcal H}}

\newcommand{\dist}{\operatorname{dist}}
\newcommand{\diam}{\operatorname{diam}}
\newcommand{\trace}{\operatorname{tr}}
\newcommand{\length}{\ell}
\newcommand{\Area}{\operatorname{Area}}
\newcommand{\Ar}{\operatorname{Ar}}

\newcommand{\md}{\operatorname{md}}
\newcommand{\lip}{\operatorname{Lip}}

\newcommand{\deltalip}{\delta^{\operatorname{Lip}}}
\newcommand{\fillarealip}{{\operatorname{Fill\,Area}^{\operatorname{Lip}}}}
\newcommand{\fillarea}{{\operatorname{Fill\,Area}}}

\newcommand{\jac}{{\mathbf J}}
\newcommand{\ap}{\operatorname{ap}}
\newcommand{\apmd}{\ap\md}

\begin{document}

\title[]{Dehn functions and Hölder extensions in asymptotic cones}

\author{Alexander Lytchak}


\address
  {Mathematisches Institut\\ Universit\"at K\"oln\\ Weyertal 86 -- 90\\ 50931 K\"oln, Germany}
\email{alytchak@math.uni-koeln.de}

\author{Stefan Wenger}

\address
  {Department of Mathematics\\ University of Fribourg\\ Chemin du Mus\'ee 23\\ 1700 Fribourg, Switzerland}
\email{stefan.wenger@unifr.ch}

\author{Robert Young}

\address{Courant Institute of Mathematical Sciences\\
  New York University\\
  251 Mercer St.\\
  New York, NY  10012\\
  USA}
\email{ryoung@cims.nyu.edu}
\date{\today}

\thanks{S.~W.~was partially supported by Swiss National Science Foundation Grant 153599.}
\thanks{R.~Y.~was partially supported by a Discovery Grant from the Natural Sciences and Engineering Research Council of Canada, 
  by a Sloan Research Fellowship, and by NSF grant DMS-1612061.}

\begin{abstract}
  The Dehn function measures the area of minimal discs that fill closed curves in a space; it is an important invariant in analysis, geometry, and geometric group theory.  There are several equivalent ways to define the Dehn function, varying according to the type of disc used.  In this paper, we introduce a new definition of the Dehn function and use it to prove several theorems.  First, we generalize the quasi-isometry invariance of the Dehn function to a broad class of spaces.  Second, we prove H\"older extension properties for spaces with quadratic Dehn function and their asymptotic cones.  Finally, we show that ultralimits and asymptotic cones of spaces with quadratic Dehn function also have quadratic Dehn function.  The proofs of our results rely on recent existence and regularity results for area-minimizing Sobolev mappings in metric spaces.
\end{abstract}

\maketitle

\section{Introduction and statement of main results}

Isoperimetric inequalities and other filling functions describe the area of minimal surfaces bounded by curves or surfaces.  Such minimal surfaces are important in geometry and analysis, and filling functions are particularly important in geometric group theory, where the Dehn function is a fundamental invariant of a group.  The Dehn function, which measures the maximal area of a minimal surface bounded by curves of at most a given length, is connected to the difficulty of solving the word problem, and its asymptotics help describe the large-scale geometry of a group.

In this paper, we study spaces with a local or global quadratic isoperimetric inequality.  The class of spaces that admit a global quadratic isoperimetric inequality includes not only spaces that satisfy a nonpositive curvature condition (such as $\delta$--hyperbolicity or the ${\rm CAT}(0)$ condition), but also spaces with a mix of positive and negative curvature, such as the higher-dimensional Heisenberg groups and many solvable Lie groups.  

Roughly speaking, spaces with a quadratic isoperimetric inequality tend to have many two-dimensional discs.  Many techniques for proving that a space has a quadratic isoperimetric inequality show that in fact the space is Lipschitz 1--connected, i.e., that curves of length $L$ are the boundaries of discs with Lipschitz constant of order $L$.  It is an open question whether a quadratic isoperimetric inequality is equivalent to Lipschitz 1--connectedness.  One way to approach this question is to study ultralimits and asymptotic cones of spaces with quadratic isoperimetric inequalities.  While Lipschitz 1--connectedness passes to asymptotic cones and ultralimits, it is an open question whether an asymptotic cone or ultralimit of a space with quadratic isoperimetric inequality also has a quadratic isoperimetric inequality.  

In \cite{Pap96} it was shown that every asymptotic cone of a finitely presented group with a quadratic Dehn function is simply connected. In \cite{Wen11-nilp} it was shown that a homological quadratic isoperimetric inequality (using metric currents) is stable under taking ultralimits and asymptotic cones. This was used in \cite{Wen11-nilp} to produce the first examples of nilpotent groups whose Dehn function does not grow exactly like a polynomial function. A weaker version of the stability was also used in \cite{Wen08-sharp} to exhibit the smallest isoperimetric constant in a quadratic isoperimetric inequality at large scales which implies that the underlying space is Gromov hyperbolic.

The aim of this article is to prove several facts about spaces with a quadratic isoperimetric inequality.  First, we will generalize the fact that the Dehn function of a manifold or simplicial complex is quasi-isometry invariant to broader classes of spaces (Theorems~\ref{thm:qi-invariance-Lip-Dehn} and \ref{thm:qi-invariance-Sobolev-Dehn}).  Second, we will show that quadratic Dehn functions are nearly stable under ultralimits (Theorem~\ref{thm:stab-isop-Lip}).  Third, we will prove Hölder extension properties for spaces with quadratic Dehn functions (Theorems~\ref{thm:Lip-Hoelder-ext-best-exponent} and \ref{thm:isop-Hoelder-ext}).  

For the convenience of the geometrically minded reader we will first formulate our results using Lipschitz maps, but the proofs rely on a generalization of the Dehn function based on Sobolev maps with values in metric spaces.  In manifolds, simplicial complexes, or other spaces that are Lipschitz 1--connected up to some scale, this Dehn function is equal to the usual Dehn function based on Lipschitz maps (Proposition~\ref{prop:Sobolev-Lip-Dehn-equal}), but this definition is better suited to some analytic arguments.  For example, we will show that an ultralimit of spaces with quadratically bounded Dehn functions has a quadratically bounded Sobolev Dehn function (Theorem~\ref{thm:stab-isop-intro}).

\subsection{The Lipschitz Dehn function and quasi-isometry invariance}

Let $(X,d)$ be a complete metric space, and let $D$ denote the open unit disc in $\R^2$ and $\overline{D}$ its closure. The (parameterized Hausdorff) area of a Lipschitz map $u\colon \overline{D}\to X$ is defined by 
\begin{equation}\label{eq:intro-def-area-Lip}
 \Area(u) = \int_X\#\{z\mid u(z) = x\} \,d\hm^2(x),
\end{equation}
 where $\hm^2$ denotes the $2$-dimensional Hausdorff measure on $X$. In particular, if $u$ is injective then $\Area(u)=\hm^2(u(D))$. If $X$ is a Riemannian manifold or simplicial complex with piecewise Riemannian metric then our definition of area coincides with the usual one which one obtains by integrating the Jacobian of the derivative of $u$, see Section~\ref{sec:prelim}. The  Lipschitz filling area of a Lipschitz curve $c\colon S^1\to X$ is given by
\begin{equation*}
 \fillarealip(c):= \inf\left\{\Area(v)\relmiddle| \text{$v\colon \overline{D}\to X$ is Lipschitz, $v|_{S^1}=c$}\right\}.
\end{equation*}
Finally, the Lipschitz Dehn (or isoperimetric) function $\deltalip_X$ of $X$ is defined by
$$\deltalip_X(r)=\sup\left\{\fillarealip(c)\relmiddle| \text{$c\colon S^1\to X$ Lipschitz, $\length(c)\leq r$}\right\}$$ for every $r>0$, where $\length(c)$ denotes the length of $c$. Thus, $\deltalip_X$ measures how the filling area of a curve depends on its length. It is important to notice that $\deltalip_X(r)$ provides an upper bound on the area of suitable fillings of a given curve of length at most $r$ but does not give control on the Lipschitz constants of such fillings. 

There are several other notions of isoperimetric functions commonly used in the setting of large scale geometry and geometric group theory. One of them is the coarse isoperimetric function $\Ar_{X, \varepsilon}(r)$ introduced by Gromov \cite{GroAII}. Roughly speaking, this function measures how difficult it is to partition a given closed Lipschitz curve into closed curves of length at most $\varepsilon$. For the precise definition of $\Ar_{X,\varepsilon}$ we refer to Section~\ref{sec:prelim} below, see also e.g.~\cite{Dru02} or \cite{BrH99}. 

Our first result below shows that under mild conditions on the metric space $X$ the functions $\deltalip_X(r)$ and $\Ar_{X,\varepsilon}(r)$ have the same asymptotic growth as $r\to\infty$. In order to state our theorem, we recall the following definition. A metric space $X$ is called Lipschitz $1$-connected up to some scale if there exist $\lambda_0>0$ and $L\geq 1$ such that every $\lambda$-Lipschitz curve $c\colon S^1\to X$ with $\lambda<\lambda_0$ extends to an $L\lambda$-Lipschitz map on $\overline{D}$. If the above holds with $\lambda_0=\infty$ then $X$ is called Lipschitz $1$-connected. Spaces that are Lipschitz $1$-connected up to some scale include Riemannian manifolds, finite dimensional simplicial complexes, and Alexandrov spaces when  equipped with a cocompact group action by isometries, as well as ${\rm CAT}(\kappa)$-spaces. 

\bt\label{thm:coarse-isop-equals-Lip-isop}
 Let $X$ be a locally compact, geodesic metric space. If $X$ is Lipschitz $1$-connected up to some scale and satisfies $\deltalip_X(r)<\infty$ for all $r>0$ then $\deltalip_X\simeq \Ar_{X,\varepsilon}$ for all $\varepsilon>0$.
\et

For the precise definition of the equivalence relation $\simeq$ for functions, see Section~\ref{sec:prelim}. 

The assumptions on $\deltalip_X$ made in the theorem are not restrictive. Indeed, every locally compact, geodesic metric space $Z$ which satisfies $\Ar_{Z,\varepsilon_0}(r)<\infty$ for some $\varepsilon_0>0$ and all $r>0$ is quasi-isometric to a space $X$ satisfying the hypotheses of Theorem~\ref{thm:coarse-isop-equals-Lip-isop}. In fact, $X$ can be obtained by suitably thickening up $Z$ and $Z$ then even embeds isometrically into $X$, see Section~\ref{sec:Sobolev-fill-fns}. In particular, for such $X$ we have $\deltalip_X\simeq \Ar_{X,\varepsilon}\simeq \Ar_{Z,\varepsilon}$ for every $\varepsilon\geq \varepsilon_0$.

It is well-known and easy to prove that the coarse isoperimetric function is a quasi-isometry invariant. Thus, Theorem~\ref{thm:coarse-isop-equals-Lip-isop} implies the quasi-isometry invariance of the Lispchitz Dehn function as well.

\bt\label{thm:qi-invariance-Lip-Dehn}
 Let $X$ and $Y$ be locally compact, geodesic metric spaces such that $\deltalip_X(r)<\infty$ and $\deltalip_Y(r)<\infty$ for all $r>0$. If $X$ and $Y$ are quasi-isometric and Lipschitz $1$-connected up to some scale then $\deltalip_X\simeq\deltalip_Y$.
\et

Theorems~\ref{thm:coarse-isop-equals-Lip-isop} and \ref{thm:qi-invariance-Lip-Dehn} generalize results in \cite{Alo90}, \cite[10.3.3]{ECHLPT92}, \cite{Bri02}, and \cite{BT02}, which proved similar results for Riemannian manifolds and simplicial complexes.  Our proof is most similar to that of Bridson in \cite{Bri02}, which proves that a geometric version of the Dehn function is equivalent to a coarse version by using a characterization of minimal surfaces in a manifold.  

\subsection{Spaces with quadratic Lipschitz Dehn function}

We now turn to metric spaces whose Lipschitz Dehn function has at most quadratic growth. A complete metric space $X$ is said to admit a quadratic isoperimetric inequality with respect to Lipschitz maps if there exists $C\geq 0$ such that $\deltalip_X(r)\leq C\cdot r^2$ for all $r\geq 0$. Any complete metric space $X$ that is Lipschitz $1$-connected admits a quadratic isoperimetric inequality.  These include Banach spaces, complete ${\rm CAT}(0)$ spaces, and, more generally, spaces with a convex geodesic bicombing. The higher Heisenberg groups (with a left-invariant Riemannian or Carnot-Carath\'eodory distance) are also Lipschitz $1$-connected.  

More generally, every complete length space $X$ that satisfies $\Ar_{X,\varepsilon}(r)\preceq r^2$ for some $\varepsilon>0$ is quasi-isometric to a complete length space $Y$ which admits a quadratic isoperimetric inequality with respect to Lipschitz maps, see Proosition~\ref{prop:quad-isop-admissible}. If $X$ is locally compact then $Y$ can be chosen locally compact as well. See Section~\ref{sec:prelim} for the definition of the relation $\preceq$.  Consequently, if $X$ is the universal cover of a closed Riemannian manifold or compact simplicial complex $M$, then $X$ admits a quadratic isoperimetric inequality with respect to Lipschitz maps if and only if $\pi_1(M)$ has a quadratic Dehn function.  

We first study H\"older extension properties of spaces with a quadratic isoperimetric inequality. Let $Z$ and $X$ be metric spaces. A map $\varphi \colon Z\to X$ is called $(\nu,\alpha)$-H\"older continuous, where $\nu\geq 0$ and $0<\alpha\leq 1$, if $$d_X(\varphi(z), \varphi(z'))\leq \nu\cdot d_Z(z,z')^\alpha$$ for all $z,z'\in Z$.  When $\alpha=1$, this is equivalent to Lipschitz continuity.

\bd
 Let $0<\alpha\leq 1$. A pair $(Z,X)$ of metric spaces $Z$ and $X$ is said to have the $\alpha$-H\"older extension property if there exists $L\geq 1$ such that every $(\nu, \alpha)$-H\"older map $\varphi\colon A\to X$ with $A\subset Z$ and $\nu>0$ admits an $(L\nu,\alpha)$-H\"older extension $\bar{\varphi}\colon Z\to X$.
\ed

If $X$ is a complete metric space and the pair $(\R^2, X)$ has the $1$-H\"older (thus Lipschitz) extension property then $X$ admits a quadratic isoperimetric inequality with respect to Lipschitz maps. The following theorem gives an "almost" converse.

\bt\label{thm:Lip-Hoelder-ext-best-exponent}
 Let $X$ be a locally compact, geodesic metric space admitting a quadratic isoperimetric inequality with respect to Lipschitz maps. Then the pair $(\R^2, X)$ has the $\alpha$-H\"older extension property for every $\alpha\in(0,1)$.
\et

For every $0<\alpha \leq 1$, the $\alpha$-H\"older extension property for $(\R^2, X)$ is stable under taking ultralimits (Corollary~\ref{cor:Hoelder-ext-property-ultralimit}).  We thus obtain a strengthening of Papasoglu's result \cite{Pap96} that every asymptotic cone of a finitely presented group with quadratic Dehn function is simply connected.

\bc\label{cor:Lip-asymp-simply-connected}
 Let $X$ be a locally compact, geodesic metric space admitting a quadratic isoperimetric inequality with respect to Lipschitz maps. Then $X$ is simply connected. Moreover, every asymptotic cone $X_\omega$ of $X$ is simply connected and the pair $(\R^2, X_\omega)$ has the $\alpha$-H\"older extension property for every $\alpha\in(0,1)$.
\ec
More generally, these two results will hold for metric spaces admitting a quadratic isoperimetric inequality with respect to Sobolev maps; see Section~\ref{sec:intro-sobolev}.

For sufficiently small $\alpha\in(0,1)$, only depending on the isoperimetric constant, we produce $\alpha$-H\"older extensions with additional properties, see Theorem~\ref{thm:isop-Hoelder-ext}. 

Our next result yields almost-stability of a quadratic isoperimetric inequality under taking asymptotic cones.


\bt\label{thm:stab-isop-Lip}
Let $X$ be a locally compact, geodesic metric space such that $\deltalip_X(r)\le Cr^2$ for all $r>0$.  Let $X_\omega$ be an asymptotic cone of $X$ and let $\varepsilon>0$. Then $X_\omega$ is $(1,\varepsilon)$-quasi-isometric to a complete length space $Y$ such that $\deltalip_X(r)\le C'r^2$ for all $r>0$, where $C'$ is a constant depending only on $C$.
\et

See Section~\ref{sec:prelim} for the definition of an $(1,\varepsilon)$-quasi-isometry. 
It is an open question whether $X_\omega$ itself admits a quadratic isoperimetric inequality with respect to Lip\-schitz maps, but we will see that it admits a slightly weaker version of the quadratic isoperimetric inequality based on Sobolev maps (see Theorem~\ref{thm:stab-isop-intro}).

As an application of Theorem~\ref{thm:stab-isop-Lip} we obtain the following special case of the main result of \cite{Wen08-sharp}. Unlike the proof in \cite{Wen08-sharp} our arguments do not rely on the theory of currents in metric spaces. 

\bt\label{thm:isop-Gromov-hyp-Lip}
 Let $X$ be a locally compact, geodesic metric space and $\varepsilon, r_0>0$. If every Lip\-schitz curve $c\colon S^1\to X$ with $\length(c)\geq r_0$ extends to a Lipschitz map $v\colon \overline{D}\to X$ with $$\Area(v)\leq \frac{1-\varepsilon}{4\pi}\cdot  \length(c)^2$$ then $X$ is Gromov hyperbolic.
\et

Theorem~\ref{thm:stab-isop-Lip} can also be used to recover results in \cite{Wen11-nilp} (whose proofs relied on the theory of metric currents) concerning super-quadratic lower bounds for the growth of the Dehn function of certain Carnot groups, see Theorem~\ref{thm:super-quad-growth} below. This result plays a role in the construction of nilpotent groups whose Dehn function do not grow exactly polynomially, see \cite{Wen11-nilp}.

\subsection{Sobolev filling functions}\label{sec:intro-sobolev}
The proofs of the results above are based on the theory of Sobolev mappings from a Euclidean domain to a complete metric space $X$ and our recent results on the existence and regularity of area-minimizing discs in proper metric spaces \cite{LW15-Plateau}, \cite{LW-intrinsic}, \cite{LW15-harmonic}.  These results let us connect Gromov's coarse isoperimetric function $\Ar_{X, \varepsilon}(r)$ to a filling function based on Sobolev maps.  

There exist various equivalent definitions of Sobolev mappings with values in metric spaces, see \cite{KS93}, \cite{Res97}, and Section~\ref{sec:prelim} below and the references therein. 
For $p>1$ the space of $p$-Sobolev maps from $D$ to $X$ will be denoted by $W^{1,p}(D, X)$.  Elements of $W^{1,p}(D, X)$ are equivalence classes of maps from $D$ to $X$, so they are defined only up to sets of measure zero.  Every Lipschitz map is $p$-Sobolev for every $p>1$, but a Sobolev map need not even have a continuous representative.  

Nevertheless, if $u\in W^{1,2}(D, X)$, then one can define a parameterized Hausdorff area of $u$, denoted by $\Area(u)$, see Section~\ref{sec:prelim}. This coincides with \eqref{eq:intro-def-area-Lip} when $u$ is Lipschitz. Moreover, if $u\in W^{1,p}(D, X)$, then there is a map $\trace(u)\in L^p(S^1,X)$, called the trace of $u$, such that if $u$ has a continuous extension $\hat{u}$ to the closed unit disc $\overline{D}$, then $\trace(u)=\hat{u}|_{S^1}$.

We can thus define a variant of the Lipschitz Dehn function $\deltalip_X(r)$ by replacing the Lipschitz filling area $\fillarealip(c)$ of a Lipschitz curve $c$ by its Sobolev variant $$\fillarea(c):= \inf\left\{\Area(v)\relmiddle| v\in W^{1,2}(D, X), \trace(v) = c\right\}.$$ Using this definition of filling area we obtain a Sobolev variant of the Dehn function which we denote by $\delta_X(r)$ and call the Dehn function of $X$, see Section~\ref{sec:Sobolev-fill-fns}. 

A complete metric space $X$ is said to admit a quadratic isoperimetric inequality (with respect to Sobolev maps) if there exists $C$ such that $\delta_X(r) \leq C\cdot r^2$ for all $r>0$. We will usually omit the phrase ``with respect to Sobolev maps'' in the sequel if there is no danger of ambiguity.  A key property of Sobolev maps is that this notion is stable under taking ultralimits and asymptotic cones.

\bt\label{thm:stab-isop-intro}
Let $C>0$, let $(X_n)$ be a sequence of locally compact, geodesic metric spaces, and let $X$ be an ultralimit of $(X_n)$.  If $\delta_{X_n}(r) \le Cr^2$ for all $r>0$, then $\delta_{X}(r) \le Cr^2$ for all $r>0$.
\et

\bc\label{cor:main-asymptotic-cones}
 Let $X$ be a locally compact, geodesic metric space. If $X$ admits a quadratic isoperimetric inequality with constant $C$ then so does every asymptotic cone of $X$. 
\ec

An analogous result for homological quadratic isoperimetric inequalities was proved in \cite{Wen11-nilp} using the theory of integral currents in metric spaces \cite{AK00}.  We mention that the isoperimetric fillings of curves in the ultralimit which we construct can moreover be taken to be globally H\"older continuous and having Lusin's property (N), that is, they send sets of Lebesgue measure zero to sets of Hausdorff $2$-measure zero.

Though $\delta_X(r)$ and $\deltalip_X(r)$ are equal in many spaces, the general relationship between the two functions is unclear.  We clearly have $\delta_X(r)\le \deltalip_X(r)$ for every complete metric space $X$, but since Sobolev maps need not have continuous representatives it is a priori not clear that a space admitting a quadratic isoperimetric inequality is even simply connected.  However, we will show that if $\delta_X(r)\le Cr^2$ for all sufficiently small $r$, then $\Ar_{X, \varepsilon}\simeq \delta_X$ for any $\varepsilon>0$ (Theorem~\ref{thm:coarse-isop-equals-Sobolev-isop}).  Similarly, if $X$ is Lipschitz $1$-connected up to some scale, then $\delta_X(r) = \deltalip_X(r)$ (Proposition~\ref{prop:Sobolev-Lip-Dehn-equal}).  These equivalences will imply Theorems~\ref{thm:coarse-isop-equals-Lip-isop}, \ref{thm:qi-invariance-Lip-Dehn}, and \ref{thm:stab-isop-Lip}.

A particular class of spaces to which the results above apply is the class of complete geodesic metric spaces $X$ which are Ahlfors $2$-regular, linearly locally contractible, and homeomorphic to $\R^2$ or $S^2$. Such spaces are of importance in various contexts, see e.g.~\cite{BK02}. It follows from the results in \cite{LW-param} that every such space $X$ admits a quadratic isoperimetric inequality (with respect to Sobolev maps).  By Theorem~\ref{thm:mesh-fct-Hoelder-ext}, every such $X$ is $\alpha$--Hölder 1--connected for every $\alpha\in (0,1)$, and thus the pair $(\R^2, X)$ has the $\alpha$-H\"older extension property. It is not known whether every such space is Lipschitz $1$-connected.

\subsection{Outline}

In Section~\ref{sec:prelim}, we will recall some definitions and facts about Sobolev maps, isoperimetric inequalities, and ultralimits.  In Sections~\ref{sec:Sobolev-fill-fns} and \ref{sec:bound-filling-mesh}, we will prove the equivalence of the Dehn function, the Lipschitz Dehn function, and the coarse isoperimetric function, showing Theorems~\ref{thm:coarse-isop-equals-Lip-isop}  and \ref{thm:qi-invariance-Lip-Dehn}.  Then, in Section~\ref{sec:stab-isop}, we show that the Dehn function is stable under ultralimits and asymptotic cones, proving Theorems~\ref{thm:stab-isop-Lip}--\ref{thm:stab-isop-intro} and Corollary~\ref{cor:main-asymptotic-cones}.  Finally, in Sections~\ref{sec:Hoelder-Sobolev}--\ref{sec:appendix}, we prove extension results for Sobolev and Hölder maps, including Theorem~\ref{thm:Lip-Hoelder-ext-best-exponent} and  Corollary~\ref{cor:Lip-asymp-simply-connected}.

\section{Preliminaries}\label{sec:prelim}

\subsection{Basic notation and definitions}

The Euclidean norm of a vector $v\in\R^n$ will be denoted by $|v|$. The unit circle in $\R^2$ with respect to the Euclidean norm is denoted by $$S^1:= \{z\in\R^2\mid |z|=1\}$$ and will be endowed with the Euclidean metric unless otherwise stated. The open unit disc in $\R^2$ is denoted by $$D:=\{z\in\R^2\mid |z|<1\},$$ its closure by $\overline{D}$.

Let $(X,d)$ be a metric space. A curve in $X$ is a continuous map $c\colon I\to X$, where $I$ is an interval or $S^1$. If $I$ is an interval then the length of $c$ is defined by $$\length(c):= \sup\left\{\sum_{i=0}^{k-1}d(c(t_i), c(t_{i+1}))\relmiddle| \text{ $t_i\in I$ and $t_0<t_1<\dots < t_k$}\right\}$$ and an analogous definition applies in the case $I=S^1$.
The space $X$ is proper if every closed ball of finite radius in $X$ is compact. For $\lambda\geq 1$ the space $X$ is $\lambda$-quasi-convex if any two points $x,y\in X$ can be joined by a curve of length at most $\lambda\cdot d(x,y)$. If $\lambda=1$ then $X$ is called geodesic. A metric space which is $\lambda$-quasi-convex for every $\lambda>1$ is called a length space. A complete, quasi-convex metric space is proper if and only if it is locally compact. 

For $s\geq 0$ the Hausdorff $s$-measure on a given metric space is denoted by $\hm^s$. We choose the normalization constant in such a way that on Euclidean $\R^n$ the Hausdorff $n$-measure coincides with the Lebesgue measure. A map from a subset of $\R^n$ to a metric space is said to satisfy Lusin's property (N) if it sends sets of Lebesgue measure zero to sets of Hausdorff $n$-measure zero.

The Lipschitz constant of a map $\varphi\colon X\to Y$ between metric space $X$ and $Y$ is denoted by $\lip(\varphi)$.

\subsection{Metric space valued Sobolev maps}
In this subsection we briefly review the main definitions and results concerning Sobolev maps from a Euclidean domain into a metric space used throughout the present paper.  There exist several equivalent definitions of Sobolev maps from Euclidean domains with values in a metric space, see e.g.~\cite{Amb90}, \cite{KS93}, \cite{Res97}, \cite{Res04}, \cite{Res06}, \cite{HKST01}, \cite{HKST15}, \cite{AT04}. Here, we recall the definition of \cite{Res97} using compositions with real-valued Lip\-schitz functions. We will restrict ourselves to Sobolev maps defined on the open unit disc $D$ of $\R^2$.

Let $(X,d)$ be a complete metric space and $p>1$. We denote by $L^p(D, X)$ the set of measurable and essentially separably valued maps $u\colon D\to X$ such that for some and thus every $x\in X$ the function $$u_x(z):= d(x, u(z))$$ belongs to $L^p(D)$, the classical space of $p$-integrable functions on $D$. 

\bd
 A map $u\in L^p(D, X)$ belongs to the Sobolev space $W^{1,p}(D, X)$ if there exists $h\in L^p(D)$ such that for every $x\in X$ the function $u_x$ belongs to the classical Sobolev space $W^{1,p}(D, X)$ and has weak gradient bounded by $|\nabla u_x|\leq h$ almost everywhere. 
\ed

The Reshetnyak $p$-energy $E_+^p(u)$ of a map $u\in W^{1,2}(D, X)$ is defined by $$E_+^p(u):= \inf\left\{\|h\|_{L^p(D)}^p\;\big|\; \text{$h$ as in the definition above}\right\}.$$

If $u\in W^{1,p}(D, X)$ then there exists a representative $\bar{u}$ of $u$ such that for almost every $v\in S^1$ the curve $t\mapsto \bar{u}(tv)$ with $t\in[1/2, 1)$ is absolutely continuous. The trace of $u$ is defined by $$\trace(u)(v):= \lim_{t\nearrow 1}\bar{u}(tv)$$ for almost every $v\in S^1$. It can be shown that $\trace(u)\in L^p(S^1, X)$, see \cite{KS93}. Clearly, if $u$ has a continuous extension $\hat{u}$ to $\overline{D}$ then $\trace(u)$ is simply the restriction of $\hat{u}$ to ${S^1}$. 

As was shown in \cite{Kar07} and \cite{LW15-Plateau}, every Sobolev map $u\in W^{1,p}(D, X)$ has an approximate metric derivative at almost every point $z\in D$ in the following sense. There exists a unique seminorm on $\R^2$, denoted $\apmd u_z$, such that 
 \begin{equation*}
    \ap\lim_{z'\to z}\frac{d(u(z'), u(z)) - \apmd u_z(z'-z)}{|z'-z|} = 0.
 \end{equation*}
Here, $\ap\lim$ denotes the approximate limit, see \cite{EG92}. If $u$ is Lipschitz then the approximate limit can be replaced by an honest limit. It can be shown, see \cite{LW15-Plateau}, that $$E_+^p(u) = \int_D\mathcal{I}_+^p(\apmd u_z)\,dz,$$ where for a seminorm $s$ on $\R^2$ we have set $\mathcal{I}_+^p(s):= \max\{s(v)^p\mid |v|=1\}$.

The (parameterized Hausdorff) area of a map $u\in W^{1,2}(D, X)$ is defined by $$\Area(u):= \int_D \jac(\apmd u_z)\,dz,$$ where the Jacobian $\jac(\|\cdot\|)$ of a norm $\|\cdot\|$ on $\R^2$ is set to be the Hausdorff $2$-measure in $(\R^2, \|\cdot\|)$ of the Euclidean unit square and $\jac(s):=0$ for a degenerate seminorm $s$ on $\R^2$. If $u\in W^{1,2}(D, X)$ satisfies Lusin's property (N) then $$\Area(u)=  \int_X\#\{z\mid u(z) = x\} \,d\hm^2(x)$$ by the area formula \cite{Kir94}. In particular, if $u$ is injective then $\Area(u) = \hm^2(u(D))$. The area and energy are related by $\Area(u)\leq E_+^2(u)$ for every $u\in W^{1,2}(D,X)$.

The following well-known properties of Sobolev maps with super-critical Sobolev exponent will be used later.

\bp\label{prop:super-critical-Sobolev}
 Let $u\in W^{1,p}(D, X)$ with $p>2$. Then $u$ has a unique representative $\bar{u}\colon\overline{D}\to X$ with the following properties:
 \begin{enumerate}
  \item $\bar{u}$ is $(L, \alpha)$-H\"older continuous on $\overline{D}$ with $\alpha = 1-\frac{2}{p}$ and $L\leq M \left[E_+^p(u)\right]^{\frac{1}{p}}$ for some $M$ depending only on $p$.
  \item $\bar{u}$ has Lusin's property (N).
\end{enumerate}
\ep

\begin{proof}
 The existence of a continuous representative $\bar{u}$ which satisfies property (i) follows from Morrey's inequality, see e.g.~\cite[Proposition 3.3]{LW15-Plateau}. For the proof of statement (ii) we refer for example to Proposition 2.4 in \cite{BMT13}.
\end{proof}

\subsection{The coarse isoperimetric function}\label{sec:coarse-isoperimetric}

We recall the definition of coarse isoperimetric function of a metric space introduced by Gromov \cite{GroAII}. Our definition is a variant of that in \cite[III.H.2.1]{BrH99}. In what follows, a triangulation of $\overline{D}$ is a homeomorphism from $\overline{D}$ to a combinatorial $2$-complex $\tau$ in which every $2$-cell is a $3$-gon. We endow $\overline{D}$ with the induced cell structure from $\tau$. For $i=0,1,2$ the $i$-skeleton of $\tau$ will be denoted $\tau^{(i)}$ and will be viewed as a subset of $\overline{D}$. The $2$-cells of $\tau$ will also be called triangles in $\tau$.

Let $X$ be a length space and $c\colon S^1\to X$ a Lipschitz curve. Let $\varepsilon>0$. An $\varepsilon$-filling of $c$ is a pair $(P, \tau)$ consisting of a triangulation $\tau$ of $\overline{D}$ and a continuous map $P\colon \tau^{(1)}\to X$ such that $P|_{S^1} = c$ and such that $\length(P|_{\partial F})\leq \varepsilon$ for every triangle $F$ in $\tau$. The $\varepsilon$-area of $c$ is defined by 
$$\Ar_\varepsilon(c):= \min\left\{|\tau| \;\big|\; \text{$(P, \tau)$ an $\varepsilon$-filling of $c$}\right\},$$ 
where $|\tau|$ denotes the number of triangles in $\tau$. If no $\varepsilon$-filling of $c$ exists we set $\Ar_\varepsilon(c):= \infty$. The $\varepsilon$-coarse isoperimetric function of $X$ is defined by $$\Ar_{X,\varepsilon}(r):= \sup\{\Ar_\varepsilon(c)\mid \text{$c\colon S^1\to X$ Lipschitz, $\length(c)\leq r$}\}$$ for all $r>0$. If $\varepsilon_0>0$ is such that $\Ar_{X, \varepsilon_0}(r)<\infty$ for all $r>0$ then for any $\varepsilon, \varepsilon'\geq \varepsilon_0$ the functions $\Ar_{X, \varepsilon}$ and $\Ar_{X, \varepsilon'}$ have the same asymptotic growth, that is, $\Ar_{X, \varepsilon}\simeq \Ar_{X, \varepsilon'}$.
Here, for functions $f,g\colon[0,\infty]\to[0,\infty)$ one writes $f\preceq g$ if there exists $C>0$ such that $$f(r)\leq Cg(Cr+C) + Cr + C$$ for all $r\geq 0$, and one writes $f\simeq g$ if $f\preceq g$ and $g\preceq f$.

If $X$ is the Cayley graph of a finitely presented group $\Gamma$ then for every sufficiently large $\varepsilon>0$ we have $\Ar_{X,\varepsilon} \simeq \delta_\Gamma$, where $\delta_\Gamma$ is the Dehn function of $\Gamma$, see \cite[III.H.2.5]{BrH99}.

Let $\lambda\geq 1$ and $\varepsilon\geq 0$. A map $\varphi\colon X\to Y$ between metric spaces $(X, d_X)$ and $(Y, d_Y)$ is a $(\lambda, \varepsilon)$-quasi-isometry if $$\lambda^{-1} d_X(x,x') - \varepsilon \leq d_Y(\varphi(x), \varphi(x'))\leq \lambda d_X(x,x') + \varepsilon$$ for all $x,x'\in X$ and for every $y\in Y$ there exists $x\in X$ with $d_Y(\varphi(x), y)\leq \varepsilon$. If such a map $\varphi$ exists then $X$ and $Y$ are called $(\lambda, \varepsilon)$-quasi-isometric. The coarse isoperimetric function is a quasi-isometry invariant, see e.g.~\cite[III.H.2.5]{BrH99}. More precisely, we have the following:

\bp\label{prop:qi-invariance-coarse-isop}
 Let $X$ and $Y$ be quasi-isometric length spaces. If there exists $\varepsilon_0>0$ such that $\Ar_{X, \varepsilon_0}(r)<\infty$ and $\Ar_{Y, \varepsilon_0}(r)<\infty$ for all $r>0$ then $\Ar_{X,\varepsilon}\simeq\Ar_{Y,\varepsilon}$ for all $\varepsilon\geq\varepsilon_0$.
\ep

\subsection{Ultralimits and asymptotic cones of metric spaces}\label{sec:prelim-ultralimit}

Let $\omega$ be a non-principal ultrafilter on $\N$, that is, $\omega$ is a finitely additive measure on $\N$ such that every subset $A\subset\N$ is $\omega$-measurable with $\omega(A)\in \{0,1\}$ and such that $\omega(\N)=1$ and $\omega(A)=0$ whenever $A$ is finite. 

Let $(Z,d)$ be a compact metric space. For every sequence $(z_n)\subset Z$ there exists a unique point $z\in Z$ such that 
$$\omega(\{n\in\N\mid d(z_n, z)>\varepsilon\}) = 0$$ for every $\varepsilon>0$. We will denote this point $z$ by $\lim\nolimits_\omega z_n$.

Let $(X_n, d_n, p_n)$ be a sequence of pointed metric spaces. A sequence of points $x_n\in X_n$ is called bounded if $$\sup_{n\in\N} d_n(x_n, p_n)<\infty.$$ One defines a pseudo-metric on the set $\tilde{X}$ of bounded sequences by $$\tilde{d}_\omega((x_n), (x'_n)):= \lim\nolimits_\omega d_n(x_n, x'_n).$$ The $\omega$-ultralimit of the sequence $(X_n, d_n, p_n)$ is defined to be the metric space obtained from $\tilde{X}$ by identifying points in $\tilde{X}$ of zero $\tilde{d}_\omega$-distance. We denote this space by $(X_\omega, d_\omega)$. An element of $X_\omega$ will be denoted by $[(x_n)]$, where $(x_n)$ is an element of $\tilde{X}$. A basepoint in $X_\omega$ is given by $p_\omega:= [(p_n)]$. Ultralimits of sequences of pointed metric spaces are always complete, see e.g.~\cite[I.5.53]{BrH99}.

We have the following relationship between pointed Gromov-Hausdorff limits and ultralimits, see \cite[Exercise I.5.52]{BrH99}. If $(X_n, d_n)$ is proper for every $n$ and converges in the pointed Gromov-Hausdorff sense to some metric space $(X_\infty, d_\infty)$ then for every non-principal ultrafilter $\omega$ the ultralimit of $(X_n, d_n, p_n)$ with respect to $\omega$ is isometric to $(X_\infty, d_\infty)$.


Let $Z$ be a metric space and $\alpha\in(0,1]$. Let $\varphi_n\colon Z\to X_n$ be $(\nu_n,\alpha)$-H\"older maps, $n\in\N$, where  $\nu_n\geq 0$ is uniformly bounded in $n$. 
If the sequence $(\varphi_n)$ is bounded in the sense that
\begin{equation*}
 \sup_{n\in\N} d_n(\varphi_n(z), p_n)<\infty
\end{equation*}
for some and thus every $z\in Z$ then the assignment $z\mapsto [(\varphi_n(z))]$ defines an $(\nu,\alpha)$-H\"older map from $Z$ to $X_\omega$ with $\nu= \lim\nolimits_\omega \nu_n$. We denote this map by $(\varphi_n)_\omega$ or $\lim\nolimits_\omega \varphi_n$. 

It follows that if $X_n$ is $\lambda_n$-quasi-convex for every $n$ and $\lambda_n$ is uniformly bounded then $X_\omega$ is $\lambda$-quasi-convex with $\lambda= \lim\nolimits_\omega \lambda_n$. In particular, ultralimits of sequences of length spaces are geodesic.

%
%

Let $(X,d)$ be a metric space, $(p_n)\subset X$ a sequence of basepoints and $(r_n)$ a sequence of positive real numbers satisfying $\lim_{n\to\infty} r_n=0$. Fix a non-principal ultrafilter $\omega$ on $\N$. The asymptotic cone of $X$ with respect to $(p_n)$, $(r_n)$, and $\omega$ is the ultralimit of the sequence $(X, r_nd, p_n)$ with respect to $\omega$. It will be denoted by $(X, r_n, p_n)_\omega$ or simply by $X_\omega$ if there is no danger of ambiguity.

\subsection{Lipschitz and H\"older maps to ultralimits}\label{sec:Lip-Hoelder}

Recall the definition of the $\alpha$-H\"older extension property for a pair $(Z,X)$ of metric spaces $Z$ and $Y$ from the introduction. When $\alpha=1$ the $\alpha$-H\"older extension property is known as the Lipschitz extension property and has been well-studied, see e.g.~\cite{Bru12} and references therein.
If a metric space $X$ is such that the pair $(\R, X)$ has the Lipschitz extension property with constant $L=\lambda$ then $X$ is $\lambda$-quasi-convex. The converse is true if $X$ is complete. In the absence of completeness one still obtains Lipschitz extensions of maps defined on closed subsets of $\R$.

\bp\label{prop:Hoelder-ultralimit}
 Let $Z$ be a separable metric space, $A\subset Z$ and $0<\alpha\leq 1$. Let $X_\omega$ be an ultralimit of a pointed sequence of metric spaces $X_n$ such that $(Z, X_n)$ has the $\alpha$-H\"older extension property with a fixed constant $L$ for all $n$. If $\varphi\colon A\to X_\omega$ is $(\nu, \alpha)$-H\"older and $\varepsilon>0$ then there exists a bounded sequence of $((1+\varepsilon)L\nu, \alpha)$-H\"older maps $\varphi_n\colon Z\to X_n$ such that the restriction of $\lim\nolimits_\omega\varphi_n$ to $A$ coincides with $\varphi$.
\ep

\begin{proof}
 We denote the metric on $Z$ by $d$ and the metric on $X_n$ by $d_n$.  Let $A\subset Z$ be a non-empty subset, $\varphi\colon A\to X_\omega$ an $(\nu,\alpha)$-H\"older map for some $\nu>0$, and let $\varepsilon>0$. Let $\{z_k\mid k\in\N\}\subset A$ be a countable dense set. For each $k\in\N$ choose a bounded sequence of points $x_{k,n}\in X_n$ such that $\varphi(z_k) = [(x_{k,n})]$. Set $$\delta_j:=\ \min\{d(z_k,z_m)\mid 1\leq k<m\leq j\}$$ for every $j\geq 2$. Define inductively a decreasing sequence $\N = N_1\supset N_2\supset \dots$ of subsets $N_j$ such that for each $j\geq 2$ we have $\omega(N_j)=1$ and $$|d_\omega(\varphi(z_k), \varphi(z_m)) - d_n(x_{k,n}, x_{m,n})| \leq \varepsilon\nu\delta_j^\alpha$$ for all $1\leq k,m\leq j$ and all $n\in N_j$.  For every $j\in\N$ define $M_j:= N_j\setminus N_{j+1}$ and set $M_\infty:= \cap_{i\in\N} N_i$. Note that the $M_j$ are pairwise disjoint and satisfy $$\N = M_\infty \cup\bigcup_{j\in\N} M_j.$$
Let $n\in\N$. If $n\in M_\infty$ then set $j(n):=n$. If $n\not\in M_\infty$ then let $j(n)$ be the unique number for which $n\in M_{j(n)}$. Define $\varphi_n\colon \{z_1, \dots, z_{j(n)}\}\to X_n$ by $\varphi_n(z_k):= x_{k,n}$ and note that $$d_n(\varphi_n(z_k), \varphi_n(z_m))\leq d_\omega(\varphi(z_k), \varphi(z_m)) + \varepsilon\nu\delta_{j(n)}^\alpha \leq (1+\varepsilon)\nu d(z_k, z_m)^\alpha$$ for all $1\leq k<m\leq  j(n)$. This shows that the map $\varphi_n$ is $((1+\varepsilon)\nu,\alpha)$-H\"older. By the $\alpha$-H\"older extension property there thus exists an $((1+\varepsilon)L\nu, \alpha)$-H\"older extension of $\varphi_n$ to the whole $Z$, again denoted by $\varphi_n$. The sequence $(\varphi_n)$ is bounded since $\varphi_n(z_1)=x_{1,n}$ for all $n\in\N$ and hence the ultralimit map $\varphi_\omega:=\lim\nolimits_\omega \varphi_n$ exists and is $((1+\varepsilon) L\nu, \alpha)$-H\"older. 

It remains to show that $\varphi_\omega = \varphi$ on $A$. Let $k\in\N$ and observe that $j(n)\geq k$ for every $n\in N_k$ and hence $\varphi_n(z_k) = x_{k,n}$ for such $n$. Since $\omega(N_k)=1$ it follows that $\varphi_\omega(z_k) = \varphi(z_k)$. Since this is true for every $k\in\N$ we conclude that $\varphi_\omega = \varphi$ on all of $A$.
\end{proof}

The following consequences of Proposition~\ref{prop:Hoelder-ultralimit} will be useful in the sequel.

\bc\label{cor:Hoelder-ext-property-ultralimit}
 Let $Z$ be a separable metric space and $0<\alpha\leq 1$. If $X_\omega$ is an ultralimit of a pointed sequence of metric spaces $X_n$ such that $(Z, X_n)$ has the $\alpha$-H\"older extension property with a fixed constant $L$ for all $n$ then $(Z,X_\omega)$ has the $\alpha$-H\"older extension property with constant $L'$ for every $L'>L$.
\ec

Since the $\alpha$-H\"older extension property is preserved under rescaling the metric we obtain furthermore that if $Z$ is separable and $X$ such that the pair $(Z,X)$ has the $\alpha$-H\"older extension property then so does the pair $(Z, X_\omega)$ for every asymptotic cone $X_\omega$ of $X$.

In view of the remark preceding Proposition~\ref{prop:Hoelder-ultralimit} and the fact that the proof of that proposition only needed extensions of maps defined on finite sets the proof also yields the following useful fact.

\bc\label{cor:lip-curve-ultralimit}
 Let $X_\omega$ be an ultralimit of some sequence of metric spaces $X_n$ each of which is $\lambda$-quasi-convex with the same $\lambda\geq 1$. Let $c\colon [0,1]\to X_\omega$ be $\nu$-Lipschitz. Then for every $\varepsilon>0$ there exists a bounded sequence of curves $c_n\colon[0,1]\to X_n$ with $c = \lim\nolimits_\omega c_n$ and such that $c_n$ is $(1+\varepsilon)\nu\lambda$-Lipschitz and satifies $$\length(c_n)\leq (1+\varepsilon)\lambda\cdot\length(c)$$ for every $n\in\N$.
\ec

The interval $[0,1]$ can be replaced by any interval $[a,b]$ or by $S^1$.

\section{Lipschitz and Sobolev filling functions}\label{sec:Sobolev-fill-fns}

Let $X$ be a complete metric space and recall the definition of the Lipschitz Dehn function $\deltalip_X$ given in the introduction. In this section, we compare $\deltalip_X$ to the (Sobolev) Dehn function $\delta_X$ defined as follows. The (Sobolev) filling area of a Lipschitz curve $c\colon S^1\to X$ is given by
\begin{equation*}
\fillarea(c):= \inf\left\{\Area(v)\relmiddle| v\in W^{1,2}(D, X), \trace(v) = c\right\}.
\end{equation*}
The (Sobolev) Dehn function of $X$ is defined by
$$\delta_X(r)=\sup\left\{\fillarea(c)\relmiddle| \text{$c\colon S^1\to X$ Lipschitz, $\length(c)\leq r$}\right\}$$ for $r>0$.  
We will usually omit the word "Sobolev" in the sequel. The following proposition shows that for many spaces $\deltalip_X$ and $\delta_X$ are equal.

\begin{prop}\label{prop:Sobolev-Lip-Dehn-equal}
  If $X$ is a complete length space that is Lipschitz $1$-connected up to some scale, then $\delta_X(r)=\deltalip_X(r)$ for all $r>0$.  
\end{prop}

\begin{proof}
It suffices to show that $\fillarea(c)=\fillarealip(c)$ for every Lipschitz curve $c\colon S^1\to X$.  Since every Lipschitz map defined on $D$ is Sobolev  we have $\fillarea(c)\le \fillarealip(c)$.  The reverse inequality follows from the proof of \cite[Theorem~8.2.1]{HKST15} and \cite[Lemma~10.1]{LW-intrinsic} by approximating a given Sobolev disc by a Lipschitz disc with almost the same area. Indeed,  let $c\colon S^1\to X$ be a Lipschitz curve and let $u\in W^{1,2}(D,X)$ be such that $\trace(u)=c$.  Let $\hat{u}\colon \bar{B}(0,2)\to X$ be the extension of $u$ to the closed disc of radius $2$ such that:
  $$\hat{u}(z)=\begin{cases}
    u(z) & z\in D\\
    c\left(\frac{z}{\|z\|}\right) & z\not\in D.
  \end{cases}$$
  Then $\Area(u)=\Area(\hat{u})$.  By the proof of \cite[Theorem~8.2.1]{HKST15}, for any $\varepsilon>0$ and any sufficiently large $t>0$, there is a set $E_t\subset B(0,3/2)$ such that $\lip(\hat{u}|_{\bar{B}(0,2)\setminus E_t})\le t$ and the Lebesgue measure of $E_t$ is at most $\frac{\varepsilon}{t^2}$. Let $M$ be the universal constant from Section~\ref{sec:appendix}, see the remark before the proof of Theorem~\ref{thm:hoelder-ext-prop-Hoelder-1-conn}. Let $\lambda_0$ and $L$ be the constants in the definition of the Lipschitz $1$-connectedness up to some scale. By taking $\varepsilon$ sufficiently small, we may assume that $E_t$ contains no balls of radius greater than $\frac{M\lambda_0}{t}$. We can now redefine $\hat{u}$ on $E_t$ to obtain a map $v\colon \bar{B}(0,2)\to X$ such that $v=\hat{u}$ outside of $E_t$ and
$$\lip(v)\le ML\cdot \lip(\hat{u}|_{\bar{B}(0,2)\setminus E_t})\le ML t,$$
see the proof of Theorem~\ref{thm:hoelder-ext-prop-Hoelder-1-conn} and the remark preceding the proof.
Then
$$\Area(v)\le \Area(\hat{u}|_{B(0,2)\setminus E_t})+\Area (v|_{E_t})\le \Area(u)+ \frac{\varepsilon}{t^2}\cdot (ML t)^2$$
and $v$ is a Lipschitz filling of $c$.  It follows that for any $\varepsilon>0$,
$$\fillarealip(c)\le \fillarea(c)+\varepsilon(ML)^2,$$
so $\fillarealip(c)\le \fillarea(c)$ as desired.
\end{proof}

We now compare $\delta_X$ and $\deltalip_X$ to the coarse isoperimetric function $\Ar_{X,\varepsilon}$. A sort of converse to the following statement will follow from Corollary~\ref{cor:Ar-bounded-by-Sobolev-isop}.

\bp\label{prop:Sobolev-isop-bounded-above-coarse-isop}
 Let $X$ be a complete length space and $\varepsilon>0$. If $L:= \delta_X(\varepsilon)<\infty$ then $$\delta_X(r) \leq L\cdot \Ar_{X,\varepsilon}(r)$$ for every $r>\varepsilon$.
\ep

The same statement holds with $\delta_X$ replaced by $\deltalip_X$. In the statement of the proposition we use the convention that $0\cdot \infty = \infty$.

\begin{proof}
 Let $r>\varepsilon$ and let $c\colon S^1\to X$ be a Lipschitz curve with $\length(c)\leq r$. We may assume that $\Ar_{X, \varepsilon}(r)<\infty$. By definition, there exists an $\varepsilon$-filling $(P, \tau)$ of $c$ with $|\tau|\leq \Ar_{X,\varepsilon}(r)$. Endow $\tau$ with the length metric such that each edge in $\tau^{(1)}$ has length $1$ and each triangle in $\tau^{(2)}$ is isometric to a Euclidean triangle. Then the space $\tau$ with this metric is biLipschitz homeomorphic to $\overline{D}$. Thus, we may assume that each triangle in $\tau$, viewed as a subset of $\overline{D}$, is biLipschitz homeomorphic to a Euclidean triangle. Let $\nu>0$ and let $F$ be a triangle in $\tau$. Since $\length(P|_{\partial F})\leq \varepsilon$ there exists a Sobolev map $u_F\in W^{1,2}(F, X)$ with $\trace(u_F) = P|_{\partial F}$ and such that $\Area(u_F)\leq \delta_X(\varepsilon) + \nu$. The gluing of all the maps $u_F$ as $F$ ranges over all triangles in $\tau$ yields a Sobolev map $u\in W^{1,2}(D, X)$ with $\trace(u) = P|_{S^1} = c$ and such that $$\Area(u) \leq |\tau|\cdot (\delta_X(\varepsilon) +\nu).$$ See \cite[Theorem 1.12.3]{KS93} for the fact that the gluing of Sobolev maps is again Sobolev. Since $\nu>0$ was arbitrary this completes the proof.
\end{proof}

Every length space is quasi-isometric to a length space that is Lipschitz $1$-connected up to some scale.  In fact, a stronger property holds.  Let $X$ and $Y$ be metric spaces and $\varepsilon>0$. We say that $Y$ is an \emph{$\varepsilon$-thickening} of $X$ if there exists an isometric embedding $\iota\colon X\to Y$ such that the Hausdorff distance between $\iota(X)$ and $Y$ is at most $\varepsilon$.  The embedding $\iota$ is then an $(1,\varepsilon)$-quasi-isometry.

\bl\label{lem:thickening-properties}
 Let $X$ be a length space.  There is a universal constant $M$ such that for every $\varepsilon>0$ there exists a complete length space $X_\varepsilon$ which is an $\varepsilon$-thickening of $X$ and has the following property. Let $\lambda>0$ and let $c_0\colon S^1\to X_\varepsilon$ be $\lambda$-Lipschitz.  If $\lambda \leq \frac{\varepsilon}{M}$, then $c_0$ is $M\lambda$-Lipschitz homotopic to a constant curve.  If $\lambda \geq \frac{\varepsilon}{M}$ and $\varepsilon'>0$, then $c_0$ is Lipschitz homotopic to a curve $c_1\colon S^1\to X$ with $\length(c_1)\leq \length(c) + \varepsilon'$ via a homotopy of area at most $M \varepsilon \lambda$.  Furthermore, if $X$ is locally compact then $X_\varepsilon$ is locally compact.
\el

In particular, $X_\varepsilon$ is Lipschitz $1$-connected up to some scale.

\begin{proof}
 The construction of the space $X_\varepsilon$ is as in the proof of \cite[Proposition 3.2]{Wen08-sharp}. The properties stated in the lemma above all follow from the proof \cite[Proposition 3.2]{Wen08-sharp}, except the very last statement. 
Since the construction in the proof uses balls in an infinite dimensional Banach space the resulting space $X_\varepsilon$ is never locally compact. In order to remedy this one replaces the balls $X_z$ in $L^\infty(B_z)$ of the spaces $B_z$ appearing in the proof by the  the injective hulls $X_z$ of $B_z$, see below. If $X$ is locally compact, then the $B_z$ and the $X_z$ are also locally compact. It follows that the resulting thickening $X_\varepsilon$ is locally compact.
\end{proof}

Injective hulls were constructed by Isbell in \cite{Isb64}.  A metric space $E$ is \emph{injective} if for every subspace $Y$ of a metric space $Z$, any 1-Lipschitz map $Y\to E$ extends to a $1$-Lipschitz map $Z\to E$.  An \emph{injective hull} $E$ of a metric space $X$ is an injective metric space equipped with an isometric embedding $e\colon X\to E$ that is minimal in an appropriate sense among injective spaces containing $X$.  We summarize, without proof, the properties of injective hulls that we used above and will use in Section~\ref{sec:stab-isop}. See \cite{Isb64} for the proofs of these properties.
\bt
  Every metric space $X$ has a unique injective hull $E(X)$ up to isometry.  This hull is contractible via a Lipschitz map.  If $X$ is compact, then $E(X)$ is compact, and if $X$ is a Banach space, then $E(X)$ is a Banach space.
\et

The Dehn function of the thickening $X_\varepsilon$ constructed above is not much bigger than that of $X$.
\bp\label{prop:Lipschitz-isoperimetric-ineq}
Let $X$ be a complete length space.  Then for every $\varepsilon>0$ there exists an $\varepsilon$-thickening $Y$ of $X$ that is a complete length space and Lipschitz $1$-connected up to some scale and satisfies $$\delta_{Y}(r) \le \varepsilon r + \delta_X(r+\varepsilon)$$ for all $r>0$.  If $X$ is locally compact, then $Y$ can be chosen to be locally compact as well.
\ep

The proof of the proposition will moreover show that if $X$ admits a quadratic isoperimetric inequality with constant $C$ then $Y$ admits a quadratic isoperimetric inequality with a constant depending only on $C$.

The following construction will be needed here and also later. We set $I:= [0,1]$.

\bl\label{lem:annulusReparam}
  There exists $M>0$ such that if $c, c'\colon  I\to X$ are
  two Lipschitz parameterizations of the same curve, then there is a Lipschitz homotopy
  $h\colon  I\times [0,1]\to X$ such that $h(s,0)=c(s)$,
  $h(s,1)=c'(s)$, and
  $\lip(h)\le M\cdot \max\{\lip(c),\lip(c')\}$.
  Moreover, $\Area(h) = 0$.
\el

The interval $I$ can be replaced by $S^1$. 

\begin{proof}
 It suffices to consider the case that $c'$ is parameterized proportional to arc-length.  Then $c'(s)=c(\varrho(s))$ for some Lipschitz function $\varrho$. Then the map
  $$h(s,t)=c((1-t) s+ t \varrho(s))$$ is a Lipschitz homotopy from $c$ to $c'$.
  The bound on the Lipschitz constant of $h$ follows by calculation. Since the image of $h$ is one-dimensional it follows that $\Area(h) = 0$.
\end{proof}

\begin{proof}[{Proof of Proposition~\ref{prop:Lipschitz-isoperimetric-ineq}}]
  Let $M$ be as in Lemma~\ref{lem:thickening-properties}. Set $\bar{\varepsilon}=\frac{\varepsilon}{2 M}$ and let $Y=X_{\bar{\varepsilon}}$ be as in Lemma~\ref{lem:thickening-properties}.  Let $c\colon S^1\to Y$ be a Lipschitz curve.  Let $\bar{c}$ be the constant speed parameterization of $c$ so that $\fillarea(c)=\fillarea(\bar{c})$ by Lemma~\ref{lem:annulusReparam}.  Note that $\bar{c}$ is $\lambda$-Lipschitz with $\lambda=\frac{\length(c)}{4}$.
If $\lambda \leq \frac{\bar{\varepsilon}}{M}$, then $\bar{c}$ is $M\lambda$-Lipschitz homotopic to a constant curve.  It follows that
$$\fillarea(\bar{c})\le 2\pi M^2\lambda^2\le 2\pi M\bar{\varepsilon}\cdot \frac{\length(c)}{4} \le \varepsilon \length(c).$$
If $\lambda\geq \frac{\bar{\varepsilon}}{M}$, then $\bar{c}$ is Lipschitz homotopic to a Lipschitz curve $c_1\colon S^1\to X$  via a homotopy of area at most $M \bar{\varepsilon} \lambda$, and $\length(c_1)\le \length(c)+\varepsilon$.  We can thus fill $c$ by a disc of area
$$\fillarea(c)\le M\bar{\varepsilon} \lambda+\fillarea(c_1)\le \varepsilon\cdot \length(c)+\delta_X(\length(c)+\varepsilon),$$
as desired.
\end{proof}

Similar thickenings allow us to work with spaces arising from geometric group theory such as Cayley graphs.

\bp\label{prop:quad-isop-admissible}
 Let $X$ be a length space such that $\Ar_{X, \varepsilon}(r)<\infty$ for some $\varepsilon>0$ and all $r>0$. Then there exists a thickening $Y$ of $X$ that is a complete length space, Lipschitz $1$-connected up to some scale, and satisfies $\deltalip_Y\preceq \Ar_{X,\varepsilon}$. If $X$ is locally compact then $Y$ can be chosen to be locally compact as well.
\ep

In particular, if $X$ is such that $\Ar_{X, \varepsilon}(r)\preceq r^2$ then there exists $C$ such that $\deltalip_Y(r)\leq Cr^2$ for all $r>0$. Similarly, if $X$ satisfies a quadratic isoperimetric inequality (with Lipschitz or Sobolev maps) only for curves of length $\geq r_0$ then there exists a thickening $Y$ of $X$ which is a complete length space, is Lipschitz $1$-connected up to some scale, and satisfies $\deltalip_Y(r)\leq Cr^2$ for some $C$ and all $r>0$.

\begin{proof}
 This is a straight-forward consequence of Proposition~\ref{prop:Sobolev-isop-bounded-above-coarse-isop} and Lemma~\ref{lem:thickening-properties}.
\end{proof}

\section{The coarse isoperimetric function}\label{sec:bound-filling-mesh}

We give bounds on the coarse isoperimetric function using the constructions in \cite{LW-intrinsic} and \cite{LW15-Plateau}. We then prove the equivalence of the coarse isoperimetric function and the Dehn function and establish the quasi-isometry invariance of the  Dehn function. This will imply Theorems~\ref{thm:coarse-isop-equals-Lip-isop}  and \ref{thm:qi-invariance-Lip-Dehn}.

Let $Z$ be a metric space homeomorphic to $\overline{D}$. We denote by $\partial Z$ the boundary circle of $Z$, that is, the image of $S^1$ under the homeomorphism from $\overline{D}$ to $Z$. Recall the definition of a triangulation of $\overline{D}$ given in Section~\ref{sec:coarse-isoperimetric}. Triangulations of $Z$ are defined in analogy.

The first result is a variant of \cite[Theorem 8.11]{LW-intrinsic}.

\bt\label{thm:triangulation-metric-discs}
 Let $Z$ be a geodesic metric space homeomorphc to $\overline{D}$ such that $\length(\partial Z)<\infty$ and $\hm^2(Z)<\infty$. Suppose there exist $C, l_0>0$ such that every Jordan domain $\Omega\subset Z$ with $\length(\partial \Omega)<l_0$ satisfies 
 \begin{equation}\label{eq:isop-domains-triangulations}
  \hm^2(\Omega)\leq C\cdot \length(\partial \Omega)^2.
 \end{equation}
Then for every $n\in\N$ with $n> \frac{8\length(\partial Z)}{l_0}$ there exists a triangulation of $Z$ into at most $$K\cdot n + K\cdot n^2\cdot\frac{\hm^2(Z)}{\length(\partial Z)^2}$$ triangles of diameter at most $\frac{\length(\partial Z)}{n}$ each, where $K$ depends only on $C$. Moreover, every edge contained in $\partial Z$ has length at most $\frac{\length(\partial Z)}{n}$.
\et

\begin{proof}
 We may assume that $C\geq 1$. By the proof of \cite[Theorem 8.11]{LW-intrinsic} there exists a finite, connected graph $\partial Z\subset \hat{G}\subset Z$ such that $Z\setminus \hat{G}$ has at most $$K\cdot n + K\cdot n^2\cdot\frac{\hm^2(Z)}{\length(\partial Z)^2}$$ components, each of which is a topological disc of diameter at most $\frac{\length(\partial Z)}{n}$. Here, $K$ only depends on the isoperimetric constant $C$. Note that \cite[Theorem 8.11]{LW-intrinsic} is formulated for a special metric space $Z$ but the proof only relies on the inequality \eqref{eq:isop-domains-triangulations}, see \cite[Remark 8.4]{LW-intrinsic}.
 We may assume that every vertex of $\hat{G}$ has degree at least $3$. We will call the components of $Z\setminus \hat{G}$ the faces of the graph $\hat{G}$. By the above, every face $F$ of $\hat{G}$ has diameter at most $\frac{\length(\partial Z)}{n}$. Denote by $v$, $e$, and $f$ the number of vertices, edges, and faces of $\hat{G}$, respectively. Notice that $e\geq \frac{3v}{2}$ and hence, by Euler's formula, $e = v+ f - 1 < \frac{2e}{3} + f$. This implies $e<3f$.

We now construct a subdivision $G$ from $\hat{G}$ as follows. Firstly, add $n$ new vertices on $\partial Z$ in such a way as to divide $\partial Z$ into $n$ segments of equal length $\frac{\length(\partial Z)}{n}$. Moreover, for each face of $\hat{G}$ whose boundary consists of a single edge, add an extra vertex in the interior of the edge. Denote this new graph by $\hat{G}$ again. The number $e'$ of edges in this new graph is bounded by $e'\leq e + n + f$. Now, fix a point in the interior of each face $F$ of $\hat{G}$ and connect it by an injective curve in $F$ to each vertex of $F$ such that the curves do not intersect except possibly at endpoints. Thus, if $F$ has $m$ vertices then $F$ is replaced by $m$ new faces. This yields a new planar graph $\partial Z\subset G\subset Z$ all of whose faces have boundary consisting of three edges and all of whose faces have diameter at most $\frac{\length(\partial Z)}{n}$. The number $L$ of faces of $G$ is bounded by the sum of degrees of vertices of $\hat{G}$, which equals $2e'$. Consequently, $$L\leq 2e'\leq 2e + 2n + 2f < 8f +2n\leq K'\cdot n + K'\cdot n^2\cdot\frac{\hm^2(Z)}{\length(\partial Z)^2}$$ for some constant $K'$ only depending on $C$. The triangulation of $Z$ associated with $G$ has the desired properties. This completes the proof.
\end{proof}

\bt\label{thm:triangulation-isop-ineq}
 Let $X$ be a locally compact, geodesic metric space such that, for some $C,r_0>0$, we have $\delta_X(r)\leq Cr^2$ for all $0<r<r_0$. Let $c\colon S^1\to X$ be a Lip\-schitz curve and $u\in W^{1,2}(D,X)$ such that $\trace(u) = c$. Then for every $n\in\N$ with $n>\frac{8\length(c)}{r_0}$ there exists a triangulation $\tau$ of $\overline{D}$ with at most $$K\cdot n + K\cdot n^2\cdot\frac{\Area(u)}{\length(c)^2}$$ triangles and a continuous map $P\colon\tau^{(1)}\to X$ such that $P|_{S^1} = c$ and $\length(P|_{\partial F})\leq \frac{4\length(c)}{n}$ for every triangle $F$ in $\tau$. Here, $K$ only depends on $C$.
\et

Note that if $X$ is Lipschitz $1$-connected up to some scale, then there exist $C,r_0>0$ such that $\delta_X(r)\leq Cr^2$ for all $0<r<r_0$.

\begin{proof}
 Consider the space $Y:= X\times \R^2$, which satisfies $\delta_Y(r)\leq C'r^2$ for all $0<r<r_0$, where $C'$ only depends on $C$, see  \cite[Lemma 3.2]{LW15-harmonic}. Let $\varepsilon>0$ be sufficiently small. Then the biLipschitz curve in $Y$ given by $\hat{c}(z):=(c(z), \varepsilon z)$ satisfies $$\length(c)<\length(\hat{c}) < \min\left\{\frac{nr_0}{8}, \frac{4\length(c)}{3}\right\}.$$
Moreover, there exists $\hat{u}\in W^{1,2}(D, Y)$ with trace $\hat{c}$ and such that $$\Area(\hat{u})\leq \Area(u) + \frac{\length(c)^2}{n},$$ compare with the proof of \cite[Theorem 3.4]{LW15-harmonic}. It follows from the results in \cite[Section 1]{LW-intrinsic} that there exist a geodesic metric space $Z$ homeomorphic to $\overline{D}$, and a $1$-Lipschitz map $\overline{u}\colon Z\to X$ with the following properties. Firstly, the restriction of $\overline{u}$ to the boundary circle $\partial Z$ of $Z$ is an arc-length preserving homeomorphism onto the image $\Gamma$ of $\hat{c}$, in particular, $\length(\partial Z)=\length(\hat{c})$. Secondly, every Jordan domain $\Omega\subset Z$ with $\length(\partial \Omega)<r_0$ satisfies $$\hm^2(\Omega)\leq C'\cdot \length(\partial \Omega)^2.$$ Thirdly, the Hausdorff measure of $Z$ is bounded by $$\hm^2(Z)\leq \Area(\hat{u})\leq \Area(u) +\frac{\length(c)^2}{n}.$$ Let $\tau$ be a triangulation of $Z$ as in Theorem~\ref{thm:triangulation-metric-discs}. Then the number $N$ of triangles in $\tau$ is bounded by $$N\leq K\cdot n + K\cdot n^2\cdot\frac{\hm^2(Z)}{\length(\partial Z)^2} \leq  (K+1)\cdot n + K\cdot n^2\cdot \frac{\Area(u)}{\length(c)^2}$$ for some constant $K$ only depending on $C$. Moreover, the diameter of each triangle in $\tau$ is bounded from above by $\frac{4\length(c)}{3n}$. Finally, the length of each edge contained in $\partial Z$ is also at most $\frac{4\length(c)}{3n}$.

Let now $\varphi\colon S^1\to \partial Z$ be the homeomorphism given by $\varphi:= (\overline{u}|_{\partial Z})^{-1}\circ\hat{c}$. By the Schoenflies theorem, $\varphi$ extends to a homeomorphism from $\overline{D}$ to $Z$, which we denote again by $\varphi$. Let $\hat{\tau}$ be the triangulation of $\overline{D}$ given by $\varphi^{-1}(\tau)$. We define a map $P\colon\hat{\tau}^{(1)}\to X$ as follows. Denote by $\pi\colon Y\to X$ the natural $1$-Lipschitz projection. For $z\in S^1$ or $z\in \hat{\tau}^{(0)}$ let $P(z):= \pi\circ \overline{u}\circ\varphi(z)$. On each edge $[v,v']$ in $\hat{\tau}^{(1)}$ not contained in $S^1$ define $P$ to be a geodesic from $P(v)$ to $P(v')$. It follows from the construction that for each triangle $F$ in $\hat{\tau}$ we have $$\length(P|_{\partial F})\leq 3\cdot \frac{4\length(c)}{3n} = \frac{4\length(c)}{n}$$ and that $P$ agrees with $c$ on $S^1$.
\end{proof}

As a consequence of Theorem~\ref{thm:triangulation-isop-ineq} we obtain a bound on the coarse isoperimetric function as follows.

\bc\label{cor:Ar-bounded-by-Sobolev-isop}
 Let $X$ be a locally compact, geodesic metric space such that, for some $C,r_0>0$, we have $\delta_X(r)\leq Cr^2$ for all $0<r<r_0$. Then there exists $K$ depending only on $C$ and $r_0$ such that $$\Ar_{X,\varepsilon}(r) \leq 1 + K\cdot r + K\cdot \delta_X(r)$$ for every $r>0$ and every $\varepsilon\geq \frac{r_0}{2}$.
\ec

\begin{proof}
 Let $\varepsilon\geq \frac{r_0}{2}$ and let $c\colon S^1\to X$ be a Lipschitz curve. Set $L:= \frac{8\length(c)}{r_0}$, and 
 let $n\in\N$ be the smallest integer with $n> L$. If $L<1$ then $\length(c)<\frac{r_0}{8}<\varepsilon$ and hence $\Ar_{X,\varepsilon}(\length(c)) = 1$. We may therefore assume that $L\geq 1$ so that $n\leq \frac{16\length(c)}{r_0}$. Since $$\frac{4\length(c)}{n}<\frac{4\length(c)}{L} = \frac{r_0}{2}\leq \varepsilon$$ it follows from Theorem~\ref{thm:triangulation-isop-ineq} that for every $\nu>0$
 \begin{equation*}
  \begin{split}
   \Ar_{X,\varepsilon}(\length(c)) & \leq K\cdot n + K\cdot n^2\cdot \frac{\delta_X(\length(c)) + \nu}{\length(c)^2}\\
   &\leq \frac{K'}{r_0}\cdot \length(c) + \frac{K'}{r_0^2}\cdot (\delta_X(\length(c)) + \nu)
  \end{split}
 \end{equation*}
 for some constants $K$ and $K'$ depending only on $C$. Since $\nu>0$ was arbitrary this completes the proof.
\end{proof}

We now obtain the following analog of Theorem~\ref{thm:coarse-isop-equals-Lip-isop}.

\bt\label{thm:coarse-isop-equals-Sobolev-isop}
 Let $X$ be a locally compact, geodesic metric space satisfying $\delta_X(r)<\infty$ for all $r>0$. If there exist $C, r_0>0$ such that $\delta_X(r)\leq Cr^2$ for all $r\in(0,r_0)$ then $\Ar_{X, \varepsilon}\simeq \delta_X$ for every $\varepsilon>0$.
\et

 This follows from Corollary~\ref{cor:Ar-bounded-by-Sobolev-isop} and Proposition~\ref{prop:Sobolev-isop-bounded-above-coarse-isop}. Theorem~\ref{thm:coarse-isop-equals-Lip-isop} is a consequence of the theorem above together with Proposition~\ref{prop:Sobolev-Lip-Dehn-equal}. Finally, the following result together with Proposition~\ref{prop:Sobolev-Lip-Dehn-equal} implies Theorem~\ref{thm:qi-invariance-Lip-Dehn}.

\bt\label{thm:qi-invariance-Sobolev-Dehn}
 Let $X$ and $Y$ be locally compact, geodesic metric spaces such that $\delta_X(r)<\infty$ and $\delta_Y(r)<\infty$ for all $r>0$ and such that there exist $C,r_0>0$ with $$\delta_X(r) \leq Cr^2\quad\text{ and }\quad \delta_Y(r)\leq Cr^2$$ for all $0<r<r_0$. If $X$ and $Y$ are quasi-isometric then $\delta_X\simeq \delta_Y$.
\et

\begin{proof}
This is a consequence of Theorem~\ref{thm:coarse-isop-equals-Sobolev-isop} together with Proposition~\ref{prop:qi-invariance-coarse-isop}. Finiteness of the functions $\Ar_{X,\varepsilon}$ and $\Ar_{Y, \varepsilon}$ follows from Corollary~\ref{cor:Ar-bounded-by-Sobolev-isop}.
\end{proof}

\section{Stability of the quadratic isoperimetric inequality}\label{sec:stab-isop}

The main result of this section is the following theorem, which readily implies Theorem~\ref{thm:stab-isop-intro} in the introduction.  

\bt\label{thm:stability-isop-ultralimit}
 Let $X_\omega$ be an ultralimit of a sequence of proper, $\lambda$-quasi-convex metric spaces $X_n$ such that, for some $C,r_0>0$, we have $\delta_{X_n}(r)\leq Cr^2$ for all $0<r<r_0$ and all $n\in\N$. Then for every Lipschitz curve $c\colon S^1\to X_\omega$ with $\length(c)<\lambda^{-1}\cdot r_0$ and every $\varepsilon>0$ there exists $u\in W^{1,p}(D, X_\omega)$ with $\trace(u)=c$ and satisfying 
 \begin{equation}\label{eq:isop-energy-bdd-ultralimit}
  \Area(u)\leq (C+\varepsilon)\lambda^2\cdot \length(c)^2\quad\text{and}\quad \left[E_+^p(u)\right]^{\frac{1}{p}}\leq C'\lambda\cdot\lip(c),
 \end{equation}
  where $p>2$ and $C'$ depend only on $C$ and $\varepsilon$. In particular, $X_\omega$ satisfies $\delta_{X_\omega}(r)\leq \lambda^2 Cr^2$ for all $0<r< \frac{r_0}{\lambda}$.
\et

Simple examples show that the constants in the isoperimetric inequality in $X_\omega$ cannot be improved in general. 

The main ingredient in the proof of the theorem is the following result obtained in \cite[Theorem 3.4]{LW15-harmonic}. 

\bt\label{thm:good-filling}
Let $X$ be a proper metric space such that, for some $C,r_0>0$, we have $\delta_{X}(r)\leq Cr^2$ for all $0<r<r_0$. Let $\varepsilon>0$. Then every Lipschitz curve $c\colon S^1\to X$ with $\length(c)< r_0$ is the trace of some $u\in W^{1,p}(D, X)$ with $$\Area(u)\leq \fillarea(c) + \varepsilon\cdot\length(c)^2 \leq (C+\varepsilon)\cdot\length(c)^2$$ and $\left[E_+^p(u)\right]^{\frac{1}{p}}\leq C'\cdot \lip(c)$, where $p>2$ and $C'$ only depend on $C$, $\varepsilon$.
\et

We can now give the proof of Theorem~\ref{thm:stability-isop-ultralimit}.

\begin{proof}
 Suppose $X_\omega$ is the ultralimit (with respect to some non-principal ultrafilter $\omega$) of the sequence of pointed metric spaces $(X_n, d_n, p_n)$.
 
 Let $c\colon S^1\to X_\omega$ be a $\nu$-Lipschitz curve of length $\length(c)<\lambda^{-1}\cdot r_0$. Fix $\varepsilon>0$ and let $\delta\in(0,1)$ be such that $(1+\delta)\lambda \length(c)<r_0$ and $(1+\delta)^2 (C+\varepsilon/2)\leq C+\varepsilon$. By Corollary~\ref{cor:lip-curve-ultralimit} there exists a bounded sequence of $(1+\delta)\lambda\nu$-Lipschitz curves $c_n\colon S^1\to X_n$ such that $\lim\nolimits_\omega c_n = c$ and such that $$\length(c_n)\leq (1+\delta)\lambda\cdot\length(c)< r_0$$ for every $n$.
 By Theorem~\ref{thm:good-filling} there exist $p>2$ and $C'$ depending only on $C$ and $\varepsilon$ such that the following holds. For each $n\in\N$ there exists $u_n\in W^{1,p}(D, X_n)$ with $\trace(u_n) = c_n$ and such that 
 \begin{equation}\label{eq:area-upper-bd-isop-stab}
  \Area(u_n)\leq \left(C+\frac{\varepsilon}{2}\right)\cdot \length(c_n)^2\leq (C+\varepsilon)\lambda^2\cdot \length(c)^2
 \end{equation}
  and $$\left[E_+^p(u_n)\right]^{\frac{1}{p}} \leq 2^{-1}C'\lip(c_n)\leq C'\lambda\nu.$$ In particular, there exists a representative of $u_n$, denoted by the same symbol, which is $(L,\alpha)$-H\"older on $\overline{D}$, where $\alpha= 1 - \frac{2}{p}$ and where $L$ depends on $C'$, $\lambda$, $\nu$ and $p$, see Proposition~\ref{prop:super-critical-Sobolev}.  
 
Now note that $(u_n)$ is a bounded sequence since each $u_n$ is $(L,\alpha)$-H\"older. Let $u\colon\overline{D}\to X_\omega$ be the ultralimit of $(u_n)$. Then $u$ is $(L,\alpha)$-H\"older on $\overline{D}$ and satisfies $u|_{S^1} = c$. In order to complete the proof it thus suffices to show that $u\in W^{1,p}(D, X_\omega)$ and that the inequalities in \eqref{eq:isop-energy-bdd-ultralimit} hold.
Observe first that the sequence of subsets $Y_n:= u_n(\overline{D})$, when endowed with the metric $d_n$, is a uniformly compact sequence of metric spaces in the sense of Gromov. This is a consequence of the fact that each $u_n$ is $(L,\alpha)$-H\"older. By Gromov's compactness theorem for metric spaces, there exist a compact metric space $(Y,d_Y)$ and isometric embeddings $\varphi_n\colon Y_n\hookrightarrow Y$ for every $n\in\N$. Let $v\colon\overline{D}\to Y$ be the ultralimit of the sequence of maps $v_n:= \varphi_n\circ u_n$. We claim that a subsequence $(v_{n_j})$ of $(v_n)$ converges uniformly to $v$ and thus also in $L^p(D, Y)$. Indeed, let $B\subset \overline{D}$ be a countable dense set. Then there exists a subsequence $(n_j)$ such that $v_{n_j}(z)\to v(z)$ as $j\to\infty$ for every $z\in B$. Since $v_{n_j}$ is $(L,\alpha)$-H\"older for every $j$ it follows that $v_{n_j}$ converges uniformly to $v$, which proves the claim. Thus, \cite[Theorem 1.6.1]{KS93} implies that $v\in W^{1,p}(D, Y)$. Moreover, $v$ satisfies the inequalities in \eqref{eq:isop-energy-bdd-ultralimit} by the weak lower semi-continuity of the $E_+^p$-energy and of area, see \cite[Corollaries 5.7 and 5.8]{LW15-Plateau}.
Finally, set $A:= v(\overline{D})$. Since $$d_Y(v(z), v(z')) = d_\omega(u(z), u(z'))$$ for all $z,z'\in\overline{D}$ there exists an isometric embedding $\iota\colon A\hookrightarrow X_\omega$ such that $u = \iota\circ v$. Since $\iota$ is isometric it follows that $u\in W^{1,p}(D, X_\omega)$ and that \eqref{eq:isop-energy-bdd-ultralimit} holds. This completes the proof.
\end{proof}

Proposition~\ref{prop:quad-isop-admissible} and Theorem~\ref{thm:stability-isop-ultralimit} imply that every asymptotic cone of a metric space with a quadratic isoperimetric inequality for sufficiently long curves admits a quadratic isoperimetric inequality with some constant which may be large.  In the following corollary we show how to obtain a better bound for the isoperimetric constant of the asymptotic cones. This will be used in the proof of Theorem~\ref{thm:isop-Gromov-hyp} below.

\bc\label{cor:isop-asymp-same-const}
 Let $X$ be a locally compact, geodesic metric space. Suppose there exist $\beta, r_0>0$ such that every Lipschitz curve $c\colon S^1\to X$ with $\length(c)\geq r_0$ is the trace of a map $u\in W^{1,2}(D, X)$ with $\Area(u)\leq \beta\cdot \length(c)^2$. Then every asymptotic cone of $X$ admits a quadratic isoperimetric inequality with constant $\beta$. 
\ec

\begin{proof}
 By Proposition~\ref{prop:quad-isop-admissible} and the remark following it, the space $X$ isometrically embeds into a locally compact, geodesic metric space $Y$ which is at finite Hausdorff distance from $X$ and which admits a quadratic isoperimetric inequality with some constant $C$. Note first that if $\hat{c}\colon S^1\to X$ is a Lipschitz curve of length $\length(\hat{c})\geq r_0$ then, by assumption, $$\fillarea_{Y}(\hat{c})\leq \fillarea_X(\hat{c}) \leq \beta\cdot\length(\hat{c})^2.$$ Here, $\fillarea_Y(\hat{c})$ and $\fillarea_X(\hat{c})$ denote the filling areas in $Y$ and $X$, respectively. This fact together with Theorem~\ref{thm:good-filling} lets us improve inequality \eqref{eq:area-upper-bd-isop-stab} in the proof of Theorem~\ref{thm:stability-isop-ultralimit} for $n$ large enough. Namely, we can find fillings $u_n$ in $Y$ which satisfy the properties in the proof of Theorem~\ref{thm:stability-isop-ultralimit} and moreover $$\Area(u_n) \leq \fillarea_{Y}(c_n) + \frac{\varepsilon}{2}\cdot \length(c_n)^2 \leq \left(\beta+\frac{\varepsilon}{2}\right)\cdot \length(c_n)^2$$ for all $n$ sufficiently large.
\end{proof}

In the rest of this section we give some consequences of the isoperimetric stability proved above. We first provide the proof of Theorem~\ref{thm:isop-Gromov-hyp-Lip}. We in fact prove the following stronger version.

\bt\label{thm:isop-Gromov-hyp}
 Let $X$ be a locally compact, geodesic metric space and $\varepsilon, r_0>0$. If every Lip\-schitz curve $c\colon S^1\to X$ with $\length(c)\geq r_0$ is the trace of some $v\in W^{1,2}(D, X)$ with $$\Area(v)\leq \frac{1-\varepsilon}{4\pi}\cdot  \length(c)^2$$ then $X$ is Gromov hyperbolic.
\et

\begin{proof}
 By \cite[Proposition 3.A.1]{Dru02} it is enough to prove that every asymptotic cone of $X$ is a metric tree. Let therefore $Y=(Y, d_Y)$ be an asymptotic cone of $X$. Since $X$ is geodesic it follows that $Y$ is geodesic.  Note also that $Y$ admits a quadratic isoperimetric inequality with some constant $C<\frac{1}{4\pi}$ by Corollary~\ref{cor:isop-asymp-same-const}.
 
We first claim that every $u\in W^{1,2}(D, Y)$ has $\Area(u)=0$. We argue by contradiction and assume there exists $u$ which has strictly positive area. It follows from \cite{Kir94} or \cite[Proposition 4.3(ii)]{LW15-Plateau} that for some $p\in Y$ and some sequence $r_n\to\infty$ the ultralimit $Y_\omega:= (Y, r_n d_Y, p)_\omega$ contains an isometric copy of a two-dimensional normed plane $V$. Let $Z$ be the injective hull of $V$. Then $Z$ is a Banach space and there exists a $1$-Lipschitz map from $Y_\omega$ to $Z$ which restricts to the identity on $V$. It follows that curves in $V$ can be filled at least as efficiently in $Z$ as in $Y_\omega$. 
Since the isoperimetric constant of every normed plane is at least $\frac{1}{4\pi}$, see \cite{Wen08-sharp}, and since affine discs in normed spaces are area minimizers by \cite{BI12} it follows that the isoperimetric constant of $Z$ and thus also of $Y_\omega$ must be at least $\frac{1}{4\pi}$. However, $Y_\omega$ is an asymptotic cone of $X$ by \cite{DS05} and thus admits a quadratic isoperimetric inequality with some constant $C<\frac{1}{4\pi}$ by Corollary~\ref{cor:isop-asymp-same-const}, a contradiction. This proves the claim.
 
We finally show that $Y$ is a metric tree. For this it suffices to show that $Y$ does not contain any rectifiable Jordan curve. We argue by contradiction and assume that there exists a rectifiable Jordan curve $\Gamma$ in $X$. Let $c\colon S^1\to \Gamma$ be a Lipschitz parameterization and let $u\in W^{1,2}(D, Y)$ be such that $\trace(u) = c$.
Then $\Area(u)=0$ by the claim above. Now, let $Z$ be the injective hull of $\Gamma$. Since $\Gamma$ is compact it follows that $Z$ is compact. Moreover, $Z$ admits a quadratic isoperimetric inequality. Denote by $\Lambda(\Gamma, Z)$ the set of all maps in $W^{1,2}(D, Z)$ whose trace is a weakly monotone parameterization of $\Gamma$. By \cite[Theorem 1.1]{LW15-Plateau} and \cite[Lemma 7.2]{LW15-Plateau}, there exists $v\in \Lambda(\Gamma, Z)\cap C^0(\overline{D}, Z)$ which satisfies $$\Area(v) = \inf\{\Area(w)\mid w\in\Lambda(\Gamma, Z)\}$$ and $E_+^2(v)\leq 2\cdot\Area(v)$.
Let $\pi\colon Y\to Z$ be a $1$-Lipschitz map which restricts to the identity on $\Gamma$. Then $w:=\pi\circ u$ is in $\Lambda(\Gamma, Z)$ and $\Area(w) = 0$. Since $v$ is an area minimizer it follows that also $\Area(v)=0$ and hence that $E_+^2(v)=0$ and hence that $v$ is constant. This contradicts the fact that $\trace(v)$ weakly monotonically parameterizes the Jordan curve $\Gamma$. 
\end{proof}

As mentioned in the introduction, one can also prove the following theorem from \cite{Wen11-nilp} using the results above. For the definitions involving Carnot groups we refer to \cite{Wen11-nilp} .

\bt\label{thm:super-quad-growth}
 Let $G$ be a Carnot group of step $2$ with grading $\mathfrak{g} = V_1\oplus V_2$ of its Lie algebra. Suppose $u\in V_2$ is such that $u \not= [v,w]$ for all $v,w\in V_1$. Let $H$ be the Carnot group of step $2$ whose Lie algebra is $\mathfrak{h} = V_1\oplus V_2/\langle u \rangle$ and endow $H$ with a left-invariant Riemannian metric. Then $H$ does not admit a quadratic isoperimetric inequality.
\et

In \cite{Wen11-nilp}, this is proved by showing that if $H$ admits a quadratic isoperimetric inequality, then minimal surfaces in $H$ converge to integral currents in its asymptotic cone.  By Pansu's Rademacher theorem for Carnot groups and Stokes' theorem for currents, there are curves in the asymptotic cone that do not bound integral currents, leading to a contradiction.  A version of Stokes' theorem also holds for Sobolev maps (with Lipschitz trace), so one can use Corollary~\ref{cor:isop-asymp-same-const} to show that minimal discs in $H$ converge to Sobolev discs and arrive at the same contradiction.

\section{H\"older-Sobolev extensions}\label{sec:Hoelder-Sobolev}

In this section we prove the first part of Corollary~\ref{cor:Lip-asymp-simply-connected} as well as Theorem~\ref{thm:isop-Hoelder-ext}, which was alluded to in the introduction and which yields H\"older extensions with additional properties. We introduce the notion of H\"older $1$-connectedness of a space and show that for quasi-convex metric spaces H\"older $1$-connectedness implies simple connectedness and is equivalent to the H\"older extension property in case the ambient space is complete. Theorem~\ref{thm:isop-Hoelder-ext} and the first part of Corollary~\ref{cor:Lip-asymp-simply-connected} will then become consequences of the results proved in Section~\ref{sec:stab-isop}.

 
 \bd
 Given $0<\alpha\leq 1$, a metric space $X$ is said to be $\alpha$-H\"older $1$-connected if there exists $L$ such that every $\nu$-Lipschitz curve $c\colon S^1\to X$ with $\nu>0$ has an $(L\nu, \alpha)$-H\"older extension to $\overline{D}$.
 \ed
 
Quasi-convex metric spaces which are $\alpha$-H\"older $1$-connected with $\alpha=1$ are called Lip\-schitz $1$-connected in \cite{LS05}. 

\bp\label{prop:Hoelder-1-conn-simply-conn}
 Let $X$ be a quasi-convex metric space. If $X$ is $\alpha$-H\"older $1$-connected for some $\alpha>0$ then $X$ is simply connected.
\ep

\begin{proof}
 Let $\lambda\geq 1$ be such that $X$ is $\lambda$-quasi-convex. We denote the metric on $X$ by $d$. Let $c\colon[0,1]\to X$ be a continuous curve with $c(0)=c(1)$. Denote by $Q$ the closed unit square in $\R^2$ and define $\varphi\colon\partial Q\to X$ by $\varphi(s,0):= c(s)$ and $\varphi(s,t):= c(0)$ if $s=0$ or $s=1$ or $t=1$. It suffices to show that $\varphi$ extends continuously to $Q$. We construct such an extension as follows. 
 For every $k\geq 0$ set $$b_k:=\max\left\{d(c(a_k^m), c(a_k^{m+1}))\relmiddle| m=0,1,\dots, 2^k-1\right\},$$ where $a_k^m:= 2^{-k}m$ for $0\leq m\leq 2^k$. On each segment $[a_k^m, a_k^{m+1}]\times\{2^{-k}\}$ with $k\geq 1$ and $0\leq m\leq 2^k-1$ define $\varphi$ to be a Lipschitz curve from $c(a_k^m)$ to $c(a_k^{m+1})$ with Lipschitz constant at most $\lambda 2^k b_k$. This defines $\varphi$ on the horizontal boundary segments of each of the rectangles $$R_k^m:= [a_k^m, a_k^{m+1}]\times [2^{-(k+1)}, 2^{-k}]$$ for $k\geq 0$ and $m=0,1,\dots, 2^k-1$. Extend $\varphi$ to $\partial R_k^m$ such that $\varphi$ is a constant map on each vertical boundary segment. Then $\varphi|_{\partial R_k^m}$ is Lipschitz with constant at most $2^kb'_k$, where $b'_k:=4\lambda(2b_{k+1} + b_k)$. Since $R_k^m$ is isometric to a $2^{-k}$-scaled copy of $R_0^0$, which is in turn biLipschitz equivalent to $\overline{D}$, it follows from the $\alpha$-H\"older $1$-connectedness that $\varphi|_{\partial R_k^m}$ extends to a $(2^{k\alpha}L'b'_k, \alpha)$-H\"older map on $R_k^m$, again denoted by $\varphi$. Here, $L'$ is a constant only depending on the constant appearing in the definition of $\alpha$-H\"older $1$-connectedness. This defines $\varphi$ on all of $Q$. Note that $$d(\varphi(z), \varphi(a_k^m)) \leq 2^\alpha L'b'_k$$ for every $z\in R_k^m$ and all $k$ and $m$. Since $c$ is continuous and thus $b'_k\to 0$ as $k\to\infty$ it follows from the inequality above that $\varphi$ is continuous at every point $(s,0)$ with $s\in [0,1]$. Since $\varphi$ is continuous on $Q\setminus ([0,1]\times\{0\})$ by construction the proof is complete.
\end{proof}

In the setting of quasi-convex metric spaces, $\alpha$-H\"older $1$-connectedness is stable under taking ultralimits. We will not need this fact and thus leave the proof to the reader.

We turn to the relationship between a quadratic isoperimetric inequality and H\"older extendability of curves. As a direct consequence of Theorems~\ref{thm:stability-isop-ultralimit} and \ref{thm:good-filling} and Propositions~\ref{prop:super-critical-Sobolev} and \ref{prop:Hoelder-1-conn-simply-conn} we obtain:

\bc\label{cor:simply-conn-Hoelder-1-conn}
 Let $X$ be a locally compact, geodesic metric space admitting a quadratic isoperimetric inequality, and let $X_\omega$ be an asymptotic cone of $X$. Then $X$ and $X_\omega$ are $\alpha_0$-H\"older $1$-connected for some $\alpha_0\in(0,1)$ depending only on the isoperimetric constant of $X$. In particular, $X$ and $X_\omega$ are simply connected.
\ec

The following result is the link between H\"older $1$-connectedness and the H\"older extension property defined in the introduction.

\bt\label{thm:hoelder-ext-prop-Hoelder-1-conn}
 Let $X$ be a complete quasi-convex metric space and $0<\alpha\leq 1$. Then $X$ is $\alpha$-H\"older $1$-connected if and only if the pair $(\R^2, X)$ has the $\alpha$-H\"older extension property.
\et

The equivalence is quantitative in the following sense. If one of the two statement holds with some constant $L\geq 1$ then the other statement holds with a constant $L'\leq M\lambda L$, where $\lambda$ is the quasi-convexity constant of $X$ and $M$ is a universal constant.
The theorem is well-known in the case $\alpha=1$, see \cite{Alm62}. We refer to the Appendix below for the proof, which is an adaptation of the arguments in \cite{Alm62}.

We conclude this section with the following result.

\bt\label{thm:isop-Hoelder-ext}
Let $X$ be a locally compact, geodesic metric space admitting a quadratic soperimetric inequality or an ultralimit of a sequence of such spaces (with uniformly bounded isoperimetric constants).
Then there exist $\alpha_0>0$ and $L\geq 1$ such that every $(\nu, \alpha)$-H\"older map $\varphi\colon A\to X$ with $\alpha\leq \alpha_0$ and $A\subset\R^2$ extends to a map $\bar{\varphi}\colon \R^2\to X$ with the following properties:
 \begin{enumerate}
  \item $\bar{\varphi}$ is $(L\nu, \alpha)$-H\"older continuous on $\R^2$.
  \item $\bar{\varphi}$ sends compact subsets of $\Omega:=\R^2\setminus \overline{A}$ to sets of finite $\hm^2$-measure.
   \item $\bar{\varphi}$ sends subsets of $\Omega$  of Lebesgue measure zero to sets of $\hm^2$-measure zero.
 \end{enumerate}
\et

The numbers $\alpha_0$ and $L$ only depend on the isoperimetric constant. If $X$ has isoperimetric constant $C=\frac{1}{4\pi}$ then we may choose $\alpha_0=1$ and $L=1$ by \cite{LW-isoperimetric} and \cite{LaS97}. Note that properties (ii) and (iii) are in general not shared by $(1-\varepsilon)$-H\"older maps, no matter how small $\varepsilon>0$. We mention that Corollary~\ref{cor:Lip-asymp-simply-connected} for $0<\alpha\leq \alpha_0$ is also a consequence of Theorem~\ref{thm:isop-Hoelder-ext}.

\begin{proof}
We begin with the following observation. Let $c\colon S^1\to X$ be a $\nu$-Lipschitz curve. By Theorems~\ref{thm:stability-isop-ultralimit} and \ref{thm:good-filling} there exists $u\in W^{1,p}(D, X)$ with $\trace(u) = c$ and such that $[E_+^p(u)]^{\frac{1}{p}}\leq C \nu$ for some $C$ and $p>2$ only depending on the isoperimetric constant of $X$. Set $\alpha_0:= 1-\frac{2}{p}$. By Proposition~\ref{prop:super-critical-Sobolev} there exists a representative $\varphi\colon \overline{D}\to X$ of $u$ which is $(M\nu, \alpha_0)$-H\"older continuous on all of $\overline{D}$ with $M$ only depending on the isoperimetric constant of $X$. Moreover, $\varphi$ satisfies Lusin's property (N). By this property and the area formula \cite{Kir94}, the Hausdorff $2$-measure of $\varphi(\overline{D})$ is finite. Finally, $\varphi$ is $(2M\nu, \alpha)$-H\"older continuous for every $0<\alpha\leq \alpha_0$.
 
Now, let $\alpha_0$ and $M$ be as above and fix $0<\alpha\leq \alpha_0$.  Let $A\subset\R^2$ be non-empty and $\varphi\colon A\to X$ a $(\nu, \alpha)$-H\"older map. Since $X$ is complete we may assume that $A$ is closed. Let $\varphi\colon \R^2\to X$ be the map constructed in the proof of Theorem~\ref{thm:hoelder-ext-prop-Hoelder-1-conn}, where $\varphi|_Q$ is a map given as in the observation above for every cube $Q$ in the Whitney cube decomposition of $\R^2\setminus A$. Then $\varphi$ is $(L\nu, \alpha)$-H\"older continuous for some $L$ depending only on $M$ by (the proof of) Theorem~\ref{thm:hoelder-ext-prop-Hoelder-1-conn}. Moreover, $\varphi$ satisfies properties (ii) and (iii) in the statement of the theorem because each $\varphi|_Q$ satisfies these properties for every $Q$ by the observation above. This completes the proof of the theorem.
\end{proof}

\section{H\"older exponent arbitrarily close to one}\label{sec:Hoelder-close-to-one}

The aim of this section is to prove the following result which, together with the results of Sections~\ref{sec:Hoelder-Sobolev} and \ref{sec:bound-filling-mesh}, yields Theorem~\ref{thm:Lip-Hoelder-ext-best-exponent} and Corollary~\ref{cor:Lip-asymp-simply-connected} and their generalizations.

\bt\label{thm:mesh-fct-Hoelder-ext}
Let $X$ be a locally compact, geodesic metric space.  Suppose that there are $C>0$ such that $\delta_X(r)\le Cr^2$ for all $r>0$.  Then, for any $\alpha\in (0,1)$, $X$ is $\alpha$-H\"older $1$-connected, with multiplicative constant depending only on $C$ and $\alpha$.
\et

Theorem~\ref{thm:hoelder-ext-prop-Hoelder-1-conn} then implies that the pair $(\R^2, X)$ has the $\alpha$-H\"older extension property, and Corollary~\ref{cor:simply-conn-Hoelder-1-conn} implies that $X$ is simply connected.  Because the multiplicative constant depends only on $C$ and $\alpha$, Corollary~\ref{cor:Hoelder-ext-property-ultralimit} implies the following corollary.

\bc\label{cor:asymp-simply-connected}
 Let $X$ be a locally compact, geodesic metric space admitting a quadratic isoperimetric inequality. Then every asymptotic cone $X_\omega$ of $X$ is simply connected and the pair $(\R^2, X_\omega)$ has the $\alpha$-H\"older extension property for every $\alpha\in(0,1)$.
\ec

We prove Theorem~\ref{thm:mesh-fct-Hoelder-ext} by using the triangulations produced by Theorem~\ref{thm:triangulation-isop-ineq} to repeatedly subdivide a curve into smaller curves.
If $X$ is a length space, we say that $X$ is $(n,k)$--triangulable if for all $r>0$, we have $\Ar_{X,r/k}(r)\le n$.  If $X$ admits a quadratic isoperimetric inequality, then Theorem~\ref{thm:triangulation-isop-ineq} implies that there is a $K$ such that $X$ is $(Kn^2,n)$--triangulable for all $n\ge 1$.
\bl\label{lem:LipschitzSubdiv}
Let $B=[0,1]^2$ be the unit square.  Let $n>0$, $k>1$.  Divide $B$ into an $n\times n$ grid of squares of side length $n^{-1}$.  Let $B_1,\dots, B_{n^2}$ be squares of side length $(2n)^{-1}$, one centered in each of the grid squares, and let $E$ be the closure of $B\setminus \bigcup_i B_i$.

There is an $L>0$ depending on $n$ such that if $X$ is a space that is $(n^2,k)$--triangulable and $c\colon \partial B\to X$ is a constant-speed Lipschitz curve, then there is a map $f\colon  E\to X$ such that:
\begin{enumerate}
\item $f$ agrees with $c$ on $\partial B$.
\item For each $i$, the restriction $f|_{\partial B_i}$ is a constant-speed curve and
  $$\length(f|_{\partial B_i})\le \frac{\length(c)}{k}.$$
\item $\lip (f)\le L \lip (c)$.
\end{enumerate}
\el

\begin{proof}
Let $c\colon \partial B\to X$ be a constant-speed Lipschitz curve.  Since $X$ is $(n^2,k)$--triangulable, there is a triangulation $\tau$ of the closed disc $\overline{D}$ with at most $n^2$ triangles, and a continuous map $P\colon\tau^{(1)}\to X$ such that $P|_{S^1} = c$ and $\length(P|_{\partial F})\leq \frac{\length(c)}{k}$ for every triangle $F$ in $\tau$.  We equip $\tau$ with a metric so that each face is isometric to a equilateral triangle with unit sides and assume that $P$ parameterizes each edge with constant speed, so that $\lip(P)\le \length(c)$.  Denote the faces of $\tau$ by $F_1,\dots, F_j$.  We will construct $f$ by finding a map $v\colon E\to \tau^{(1)}$ that sends $\partial B$ to $\partial \tau$ and sends $\partial B_i$ to $\partial F_i$ for all $i\le j$.

For all $i$, let $x_i$ be the center of the square $B_i$.  For $1\le i\le j$, let $y_i$ be the center of the face $F_i$, and for $i>j$, let $y_i$ be a point in the interior of $\tau$, chosen so that the $y_i$ are all distinct.  Let $q\colon B\to B$ be a Lipschitz map that fixes the first $j$ grid squares of $B$ and collapses $\partial B_i$ to $x_i$ for $i>j$.  Let $g\colon B\to \tau$ be a Lipschitz homeomorphism such that for all $i$, $g(x_i)=y_i$ and $g(\partial B_i)$ is a small circle around $y_i$.  Then $g\circ q(E)$ avoids the center of each face of $\tau$.

Let 
$$r\colon  \tau \setminus \{y_1,\dots, y_j\}\to \tau^{(1)}$$
be the retraction such that if $x\in F_i$, then $f(x)$ is the intersection of the ray from $y_i$ to $x$ with the boundary $\partial F_i$.  This map is locally Lipschitz, so $v=r\circ g\circ q$ is a Lipschitz map from $E$ to $\tau^{(1)}$ such that $v(\partial B_i)=\partial F_i$ for all $i\le j$ and $v(\partial B_i)$ is a point for all $i>j$.

The Lipschitz constant of $v$ depends on $\tau$, but there are only finitely many possible triangulations $\tau$ with at most $n^2$ faces.  Each of these gives rise to a map $v_\tau$, so if $L_0=L_0(n)=\max_\tau \lip(v_\tau)$, then $\lip(v)\le L_0$.  Let $f_0=P\circ v$.  Then $\lip(f_0)\le L_0 \lip(P)\le L_0\length(c)$.

This map sends $\partial B$ to $c$ and $\partial B_i$ to either a curve of length at most $\frac{\length(c)}{k}$ or a point, but the parameterizations of the boundary components may not have constant speed.  Let $E'$ be the space $E$ with an annulus of length $\frac{1}{n}$ glued to each boundary component $\partial B_i$ and an annulus of length $1$ glued to $\partial B$.  This is biLipschitz equivalent to $E$, say by an equivalence $h\colon E'\to E$.  Define $f\colon E\to X$ so that $f|_{\partial B}=c$, $f|_{\partial B_i}$ is a constant-speed parameterization of $f_0|_{\partial B_i}$, and $f(h(x))=f_0(x)$ for all $x\in E$.  This defines $f$ on all of $E$ except for an annulus around each boundary component, and we use Lemma~\ref{lem:annulusReparam} to define $f$ on these annuli.  Then $f$ satisfies the desired conditions on each boundary component and there is a $L_1>0$ such that $\lip(f)\le L_1\lip(f_0)\le L_1L_0\length(c)$, as desired.
\end{proof}

Now, we can prove the desired H\"older extension property.
\bp\label{prop:subdivHolder}
If $X$ is $(n^2,k)$--triangulable, then for any Lipschitz curve $c\colon  S^1\to X$, there is an extension $\beta\colon  \overline{D}\to X$ such that $\beta$ is $\left(C_H \lip(c), \frac{\log k}{\log 2n}\right)$-H\"older continuous, where $C_H$ depends only on $n$ and $k$
\ep
\begin{proof}
We identify $S^1$ with the boundary $\partial B$, where $B=[0,1]^2$.  Let $E^0$ be $B$ with $n^2$ holes of side length $(2n)^{-1}$ removed, as in Lemma~\ref{lem:LipschitzSubdiv}.  We denote these holes by $B^1_1,\dots B^1_{n^2}$.

Let $c\colon  \partial B\to X$ be a Lipschitz curve. By Lemma~\ref{lem:annulusReparam}, we may assume that $c$ is parameterized with constant speed.  Let $\beta^0\colon  E^0\to X$ be a map satisfying Lemma~\ref{lem:LipschitzSubdiv}.  The curves $c_i=\beta^0|_{\partial B^1_i}$ are Lipschitz curves with $\length(c_i)\le k^{-1}\length(c).$
Since $c_i$ is a constant-speed curve, we have $\lip (c_i)\le 2n k^{-1}\length(c).$
For $i=1,\dots, n^2$, let $E^1_i$ be a copy of $E^0$, scaled to fit inside $B^1_i$.  We can apply Lemma~\ref{lem:LipschitzSubdiv} to each curve $c_i$ to obtain maps $\beta^1_i\colon  E^1_i\to X$ that agree with $\beta^0$ on $\partial B^1_i$ and satisfy
$$\lip (\beta^1_i)\le  L\lip(c_i)\le 2n Lk^{-1}\length(c).$$
Let $E^1=E^0\cup \bigcup_i E^1_i$ and let $\beta^1\colon  E^1\to X$ be the map that extends $\beta^0$ and all of the $\beta^1_i$'s.

Now, $E^1$ has $n^4$ smaller square holes, which we denote
$B^2_1,\dots B^2_{n^4}$.  We can repeat the process to produce
$E^2_i$, $\beta^2_i$, etc.  In fact, for each $j>0$ and
$i=1,\dots, n^{2j}$, there are $B^j_i$, $E^j_i$, and
$\beta^j_i\colon  E^j_i\to X$ with the following properties:
\begin{enumerate}
\item The set $B^j_{i}$ is one of the holes in
  $E^{j-1}_{\lceil i/n^2\rceil}$; it has side length $(2n)^{-j}$.
\item $E^j_i$ is a copy of $E^0$, scaled to fit in $B^j_i$.
\item $\beta^j_i\colon  E^j_i \to X$ is a map such that $\beta^j_i$
  agrees with $\beta^{j-1}_{\lceil i/n^2\rceil}$ on $\partial B^j_i$.
\item $\length(\beta^j_i|_{\partial B^j_i})\le  k^{-j}\length(c).$
\item \label{enum:LipschitzBeta} 
  $\lip (\beta^j_i) \le (2n)^j L  k^{-j}\length(c).$
\end{enumerate}
Let $E^m=\bigcup_{j=1}^m \bigcup_{i=1}^{n^{2j}} E^j_i$ and let
$E^\infty=\bigcup_{j=1}^\infty E^j$.  Then $B\setminus E^\infty$ is a Cantor set of measure zero.  Let $\beta^j\colon E^j\to X$ be the common extension of the $\beta^m_i$ for $m\le j$.  By (\ref{enum:LipschitzBeta}), $\lip (\beta^j) \le (2n)^j L  k^{-j}\length(c).$  Let $\beta^\infty\colon  E^\infty\to X$ be the common extension of the maps $\beta^j$.  Then $\beta^\infty$ is an extension of $c$, and we claim that $\beta^\infty$ is H\"older continuous.

  Suppose that $x\in E^\infty$.  For each $j$, let $x_j\in E^j$ be a
  point such that $|x-x_j|=\dist(x,E^j)$.  The holes in $E^j$ have side
  length $(2n)^{-j-1}$, so $|x - x_j|\le (2n)^{-j-1}$.  Then
  \begin{align*}
    d(\beta^{\infty}(x),\beta^{\infty}(x_j))
    &\le \sum_{i=j}^{\infty}\lip(\beta^{i+1})\cdot d(x_{i},x_{i+1})\\
    &\le \sum_{i=j}^{\infty} (2n)^{i+1} L k^{-i-1}\length(c) \cdot
      [(2n)^{-i-1}+(2n)^{-i-2}]\\
    &\le \frac{1+(2n)^{-1}}{1-k^{-1}}\cdot L k^{-(j+1)}\length(c)\\
    &\le L' k^{-j}\length(c),
  \end{align*}
  where $L'$ is a constant depending on $n$ and $k$.

  Suppose that $u,v\in E^\infty$, and let $m\ge 0$ be such that
  $$\sqrt{2}(2n)^{-m-1}\le |u - v|\le \sqrt{2}(2n)^{-m}.$$
  Let $u_m, v_m\in E^m$ be the nearest-point projections of $u$
  and $v$.  Then
  $$|u_m - v_m|\le \sqrt{2}(2n)^{-m}+2(2n)^{-m-1}\le 3(2n)^{-m}$$
  and
  \begin{align*}
    d(\beta^{\infty}(u),\beta^{\infty}(v))
    &\le 2 L' \length(c)k^{-m}+d(\beta^{\infty}(u_m),\beta^{\infty}(v_m)) \\ 
    &\le 2 L' \length(c)k^{-m}+3\lip(\beta^m) (2n)^{-m} \\
    &\le 2 L' \length(c)k^{-m}+3L \length(c) k^{-m}\\
    &\le (2 L'+3L)\cdot \length(c)\cdot k^{-m} \\
    &\le C_H \length(c)\cdot |u - v|^{\frac{\log k}{\log 2n}}
  \end{align*}
  where $C_H$ depends on $n$ and $k$.  Therefore, $\beta^{\infty}$ is
  $\frac{\log k}{\log 2n}$-H\"older continuous, and so there exists a unique $\frac{\log k}{\log 2n}$-H\"older continuous map $\beta\colon  B\to X$ that extends $c$.
\end{proof}

\begin{proof}[Proof of Theorem~\ref{thm:mesh-fct-Hoelder-ext}]
Let $\alpha\in (0,1)$.  By Theorem~\ref{thm:triangulation-isop-ineq}, there is a $K$ depending only on $C$ such that $X$ is $(Kn^2,n)$--triangulable for all $n\ge 1$.  We have
$$\lim_{n\to \infty}\frac{\log n}{\log 2\sqrt{Kn^2}}=1,$$
so there is an $n$ such that $\frac{\log n}{\log 2\sqrt{Kn^2}}>\alpha$.  By Proposition~\ref{prop:subdivHolder}, $X$ is $\frac{\log n}{\log 2\sqrt{Kn^2}}$--Hölder $1$--connected with a constant depending only on $K$ and $n$.
\end{proof}

\section{Appendix: Proof of Theorem~\ref{thm:hoelder-ext-prop-Hoelder-1-conn}}\label{sec:appendix}

In this section, we prove Theorem~\ref{thm:hoelder-ext-prop-Hoelder-1-conn}. The proof is a variation of the proof of Theorem (1.2) in \cite{Alm62}, which is the special case of our theorem for $\alpha=1$. The construction of our H\"older continuous extension is exactly as in \cite{Alm62}, while the proof of the H\"older continuity is slightly more technical than the proof of the Lipschitz continuity in \cite{Alm62}. 

The proof of the theorem will moreover show the following. If $X$ only satisfies the $\alpha$-H\"older $1$-connectedness condition for Lipschitz curves with Lipschitz constant at most some number $\nu_0$ then $(\nu, \alpha)$-H\"older maps $\varphi\colon A\to X$ with $A\subset\R^2$ and $\nu>0$ arbitrary can be extended to $(ML\nu, \alpha)$-H\"older maps defined on the neighborhood $$\left\{z\in \R^2\relmiddle| \dist(z, A)< \left(\frac{M\nu_0}{\lambda\nu}\right)^{\frac{1}{\alpha}}\right\}$$ of $A$, where $\lambda$ is the quasi-convexity constant and $M$ is a universal constant.

\begin{proof}[Proof of Theorem~\ref{thm:hoelder-ext-prop-Hoelder-1-conn}]
 Suppose first that the pair $(\R^2, X)$ has the $\alpha$-H\"older extension property with some constant $L$.  Let $c\colon S^1\to X$ be a $\nu$-Lipschitz curve. Then $c$ is $(2\nu, \alpha)$-H\"older and thus, by the H\"older extension property, there exists an $(2\nu L, \alpha)$-H\"older extension of $c$ defined on all of $\R^2$. It follows that $X$ is $\alpha$-H\"older $1$-connected with constant $L'=2L$.
 
 Now, suppose that $X$ is $\alpha$-H\"older $1$-connected with some constant $L\geq 1$.  In what follows, $M$ will denote a suitable universal constant, which might change with each occurrence in the text. Denote by $\lambda$ the quasi-convexity constant of $X$. Let $A\subset\R^2$ and let $\varphi\colon A\to X$ be $(\nu, \alpha)$-H\"older continuous for some $\nu>0$. Since $X$ is complete we may assume that $A$ is closed. Let $\mathcal{Q}$ be a Whitney cube decomposition of the complement $A^c$ of $A$, see Theorem VI.1.1 in \cite{Ste70}. Thus, $\mathcal{Q}$ is a collection of pairwise almost disjoint closed squares $Q\subset A^c$ of the form $$Q=[p,p+2^k]\times [q, q+2^k]$$ with $p,q\in 2^k\Z$ and $k\in\Z$ such that $A^c= \cup_{Q\in\mathcal{Q}}Q$ and 
 \begin{equation}\label{eq:property-Whitney-cube-ineq}
  4^{-1} \diam(Q)\leq \dist(A, Q)\leq 4\diam(Q)
 \end{equation}
 for every $Q\in \mathcal{Q}$. We first observe that if $Q,Q'\in\mathcal{Q}$ have non-trivial intersection then $\diam(Q)\leq 20 \diam(Q')$ by \eqref{eq:property-Whitney-cube-ineq}. Moreover, either $Q\cap Q'$ consists of a single point or it contains an entire edge of $Q$ or $Q'$.  
Denote by $K_0$ the set of vertices and by $K_1$ the union of edges of squares in $\mathcal{Q}$. An edge of a square in $\mathcal{Q}$ will be called minimal edge in $K_1$ if it does not contain any shorter edges of squares in $\mathcal{Q}$. Note that for every $Q\in \mathcal{Q}$ and every minimal edge $I$ in $K_1$ with $I\subset \partial Q$ we have 
\begin{equation}\label{eq:min-edge-long}
 \diam(I)\geq (\sqrt{2}\cdot 20)^{-1}\diam(Q)
\end{equation}
 by the above.

We extend $\varphi$ first to points in $K_0$ as follows. For each $x\in K_0$ choose $r_0(x)\in A$ such that $$|x-r_0(x)| = \dist(x, A)$$ and set $\varphi(x):= \varphi\circ r_0(x)$. We next extend $\varphi|_{K_0}$ to $K_1$ using the quasi-convexity of $X$. For this, let $I\subset K_1$ be a minimal edge in the sense above and let $a$ and $b$ denote the endpoints of $I$. It follows from \eqref{eq:property-Whitney-cube-ineq} that $$|r_0(a) - r_0(b)|\leq M|a-b|$$ and hence $$d(\varphi(a), \varphi(b)) \leq \nu |r_0(a) - r_0(b)|^\alpha \leq \nu M |a-b|^\alpha.$$ Thus, $\varphi$ extends to some $(M\lambda\nu  \diam(I)^{\alpha-1})$-Lipschitz map on $I$ by the $\lambda$-quasi-convexity of $X$. Since $K_1$ is the union of minimal edges in $K_1$ this defines $\varphi$ on all of $K_1$. We finally extend $\varphi|_{K_1}$ to $A^c$ using the $\alpha$-H\"older $1$-connectedness. For this, let $Q\in \mathcal{Q}$. It follows with \eqref{eq:min-edge-long} that $\varphi|_{\partial Q}$ is $(M\lambda\nu \diam(Q)^{\alpha-1})$-Lipschitz. Hence $\varphi|_{\partial Q}$ extends to some $(ML\lambda\nu, \alpha)$-H\"older map defined on $Q$. Since $Q$ was arbitrary this defines $\varphi$ on all of $A^c$ and hence on all of $\R^2$.

It remains to show that $\varphi$ is $(ML\lambda\nu,\alpha)$-H\"older continuous on all of $\R^2$. Let $x,y\in\R^2$. We first assume that $x\in A$ and $y\in A^c$. Let $Q\in\mathcal{Q}$ be such that $y\in Q$ and let $y'$ be a vertex of $Q$. Then $$|y-y'|\leq \diam(Q)\leq 4\dist(Q,A)\leq 4|x-y|$$ and hence 
\begin{equation*}
 \begin{split}
  |r_0(y') - x|&\leq |r_0(y') - y'| + |y'-y| + |y-x|\\
  &\leq \dist(y', A) + 5|y-x|\\
  &\leq \dist(Q, A) + \diam(Q) + 5|y-x|\\
  &\leq 5\dist(Q,A) + 5|y-x|\\
  &\leq 10|y-x|.
 \end{split}
\end{equation*}
From this we infer that 
\begin{equation*}
\begin{split}
 d(\varphi(x), \varphi(y)) &\leq d(\varphi(x), \varphi(y')) + d(\varphi(y'), \varphi(y))\\
 &\leq \nu|x-r_0(y')|^\alpha + ML\lambda\nu |y'-y|^\alpha\\
 &\leq ML\lambda\nu |x-y|^\alpha.
\end{split}
\end{equation*}
We now assume that $x,y\in A^c$. Let $Q,Q'\in\mathcal{Q}$ be such that $x\in Q$ and $y\in Q'$ and let $x'$ and $y'$ be vertices of $Q$ and $Q'$, respectively. We distinguish two cases and first consider the case that 
\begin{equation*}
|x-y|\geq \frac{1}{30}\max\{\dist(Q,A), \dist(Q',A)\}.
\end{equation*}
From this and \eqref{eq:property-Whitney-cube-ineq} we infer that $|x-x'|\leq 120|x-y|$ and $|y-y'|\leq 120|x-y|$ and moreover
\begin{equation*}
  |r_0(x') - r_0(y')| \leq M |x-y|.
\end{equation*}
It follows that
 \begin{equation*}
  \begin{split}
   d(\varphi(x), \varphi(y)) &\leq d(\varphi(x), \varphi(x')) + d(\varphi(x'), \varphi(y')) + d(\varphi(y'), \varphi(y))\\
    &\leq ML\lambda\nu (|x-x'|^\alpha + |y-y'|^\alpha) +  \nu|r_0(x') - r_0(y')|^\alpha\\
    &\leq ML\lambda\nu |x-y|^\alpha,
  \end{split}
 \end{equation*}
which proves the first case. We now consider the case that 
 \begin{equation*}
|x-y|< \frac{1}{30}\max\{\dist(Q,A), \dist(Q',A)\}.
\end{equation*}
We may assume without loss of generality that $x\not=y$ and $\dist(Q',A)\leq \dist(Q,A)$. Denote by $Z$ the square of side length $2|x-y|$ centered at $x$. We claim that at most $4$ squares from $\mathcal{Q}$ intersect $Z$. Indeed, if $Q''\in\mathcal{Q}$ intersects $Z$ in some point $z$ then $$\diam(Q'')\geq \frac{1}{4}\dist(Q'', A)\geq  \frac{1}{4}(\dist(z,A) - \diam(Q''))$$ and thus $\diam(Q'')\geq \frac{1}{5}\dist(z,A)$. Since $$\dist(z,A)\geq \dist(x,A) - |x-z|\geq \dist (Q,A)-|x-z|> (30-\sqrt{2})\cdot |x-y|$$ it follows that $\diam(Q'')\geq 5^{-1}(30 -\sqrt{2})\cdot |x-y|$ and hence the side length of $Q''$ is strictly larger than the side length of $Z$. This implies that $Z$ intersects at most $4$ squares from $\mathcal{Q}$, as claimed. It follows that the straight segment from $x$ to $y$ intersects at most $4$ squares from $\mathcal{Q}$. Since $\varphi$ is $(ML\lambda\nu, \alpha)$-H\"older on each such square it follows that $$d(\varphi(x), \varphi(y))\leq 4ML\lambda\nu |x-y|^\alpha.$$ This completes the proof.
\end{proof}

\def\cprime{$'$} \def\cprime{$'$} \def\cprime{$'$}

\end{document}